\documentclass[smallextended]{svjour3}       
\smartqed  
\usepackage{graphicx}
\usepackage{mathptmx}      
\usepackage{graphicx}
\usepackage[numbers,sort&compress]{natbib}
\usepackage{latexsym}
\usepackage{amscd}
\usepackage{amssymb}
\usepackage{amstext}
\usepackage{amsmath}
\usepackage{amsxtra}
   \usepackage[hang,centerlast,small]{subfigure}
      \newcommand*{\ud}{\hspace{0.1cm}\mathrm{d}}
      
      \newcommand{\bR}{\mathbb{R}}
      
      \newcommand{\fm}{E}
      \newcommand{\es}{{S_\lambda}}
      \newcommand{\E}{{\cal E}}
      \newcommand{\m}{{\cal M}}
      
      \newcommand{\ct}{{\cal T}}
      \newcommand{\cl}{{\cal L}}
      \newcommand{\cc}{{\cal C}}
      
      \newcommand{\cq}{{\cal Q}}
      \newcommand{\F}{{\cal F}}
      
      \newcommand{\wt}{\widetilde}
   \def\nn{\nonumber}
   \def\ni{\noindent}
   \def\!{\mskip-0.6\thinmuskip}
   \newcommand{\snorm}[1]{
   \bgroup\left\vert#1\right\vert\egroup}
   \newcommand{\dnorm}[1]{
   \bgroup\left\vert\!\left\vert#1\right\vert\!\right\vert\egroup}
   \newcommand{\tnorm}[1]{
   \bgroup\left\vert\!\left\vert\!\left\vert#1\right\vert\!\right\vert\!\right\vert\egroup}
   \newcommand{\iprod}[2]{\bgroup\left<#1 , #2\right>\egroup}
   \def\nn{\nonumber}
   \def\ni{\noindent}
 \usepackage{mathptmx}      
\begin{document}

\title{The nonlinear fractional diffusion equations with Nagumo-type sources and  perturbed orders
}


\author{Nguyen Minh Dien \and Erkan Nane \and and Dang Duc Trong
}


\institute{
              Nguyen Minh Dien
              \at
              Faculty of Natural Sciences,
              Thu Dau Mot University, Thu Dau Mot City, Binh Duong Province, Viet Nam
             \\
Department of Mathematics and Computer Science,
              University of Science, Vietnam National University HCMC, Ho Chi Minh City, Viet Nam
 \\
              \email{diennm@tdmu.edu.vn}
              \\
~
              \and
              Erkan Nane
              \at
              Department of  Mathematics and Statistics, Auburn University, Alabama 36849, USA
              \\
              \email{ezn0001@auburn.edu}
              ~
              \and
              Dang Duc Trong
              \at
              Department of Mathematics and Computer Science,
              University of Science, Vietnam National University HCMC, Ho Chi Minh City, Viet Nam
              \\
              \email{ddtrong@hcmus.edu.vn}
              }

\date{Received: date / Accepted: date}

\maketitle

\begin{abstract}
We consider a class of nonlinear fractional equations having
the Caputo fractional derivative of the time variable $t$, the fractional order of the self-adjoint positive definite unbounded operator in a Hilbert space and a singular nonlinear source.
These  equations are generalizations  of  some well--known fractional equation such as the  fractional
Cahn--Allen equation, the fractional Burger equation, the fractional Cahn--Hilliard equation,
the fractional Kuramoto--Sivashinsky equation, etc.
~
We study both the initial value and the final value problem.
~
Under some suitable assumptions, we investigate the existence, uniqueness of maximal solution, and stability of
solution of the problems with respect to perturbed fractional orders.
~
For $t=0$, we show that the final value problem is instable and deduce that the problem is ill-posed.
A regularization method is proposed to recover the initial data from the inexact fractional orders and the final data.
By some regularity assumptions of the exact solutions of the problems, we obtain an
error estimate of  H\"older type.

\end{abstract}

\section{Introduction}

\subsection{Statement of the problem}
Let $H$ be a Hilbert space,
$A : D(A) \subset H \to H$ be a self-adjoint positive definite unbounded operator and
$f: [0, +\infty) \times \mathcal{B} \to L^2((0, \infty), H)$ with $\mathcal{B}\subset H$.
For $\alpha \in (0, 1]$, $\beta>0$,
we consider the problem of finding a function $u: [0,T]\to H$ satisfying
~
\begin{equation}\label{main}
D_t^\alpha u  +   A^\beta u
=
f(t, u(t)), \, \, t > 0,
\end{equation}
~
~
where
 $  D_t^\alpha$
is the Caputo fractional derivative
\[
D_t^\alpha u (t)
=
\left\{
\begin{array}{lll}
 &
     \frac{1}{\Gamma(1-\alpha)}
     \int_0^t (t-s)^{-\alpha}u'(s) \ud s,
 &
     \textrm{if} \, \, \alpha\in (0, 1),
    \\
  &
     u'(t),
  &
    \textrm{if} \, \, \alpha=1,
\end{array}
\right.
\]
and the power of the operator $A^\beta$ will be defined later.
The equation \eqref{main} is a general form of a lot of well-known equations as
the  Ginzburg--Landau equation ($ \alpha=1, A=-\Delta, f(t,u)=au-bu^3$),  Burger equation
($ \alpha=1, A=-\Delta,  f(t,u)=u u_x$),  Kuramoto--Sivashinsky equation ($ \alpha=1, A=\Delta^2,  f(t,u)=\nabla
^2 u + (1/2)\| \nabla u \|^2)$).
In the present paper, we will investigate the stability of solution of the initial value  and the final value problems of \eqref{main}.
As known, the equation (\ref{main}) subjects to the initial data
\begin{equation}
\label{initial value}
  u(0)=\zeta
\end{equation}
is called the fractional initial value (or the Cauchy, the forward) problem (FIVP)
and the problem (\ref{main}) subject to the final data
\begin{equation}
\label{final value}
  u(T)=\varphi
\end{equation}
is called the fractional final value (or backward) problem (FFVP).

\subsection{History and motivation}
The abstract parabolic equations $u_t+Au=f$ was considered for thirty years with  a lot of papers,
readers can see the classical book by Cazenave and Haraux \cite{cazenave1998introduction} and references therein.
The FIVP  were also studied in a lot of papers.
Xing et al \cite{ChengLiYamamoto} discussed the existence, uniqueness, analyticity and the long-time asymptotic behavior of solutions of space-time
fractional reaction--diffusion equations in $\bR^n$
\[
   D_t^\alpha u +(- \Delta)^\beta u =p(x) u,
\]
subject to the initial condition $u(x, 0) =a(x)$.
Existence and uniqueness of the maximal solutions of some linear and nonlinear fractional problems were investigated in
\cite{Clement, Duan, Foias}. The blow--up and global solution of  time--fractional nonlinear
diffusion--reaction equations were studied recently with some kind of nonlinear sources  as
$f(t,u)=au+u^p$ (see Cao et al \cite{cao2019}),  $f(t,u)=|u|^p$
 (Zhang \cite{Zhang}),  Asogwa  et al. \cite{asogwa-foondun-mijena-nane-2018} studied  finite time blow up results for a variation  of equation \eqref{main}.
In many practical situation, the source is often assumed to satisfy
\[
   \dnorm{f(t,u)-f(t,v)}
\le
   K(t,M)\| u-v\|
   \ \ \ {\rm for \ } \| u\| ,\| v\| \le M,
\]
where $K:(0,T)\times (0,\infty)\to \mathbb{R}$. Generally, the source $f$ is assumed to be  locally Lipschitz with respect to the variable $u$, i.e., $\sup_{t\in (0,T)}|K(t,M)|<\infty$.
However, a singular source satisfying $\sup_{t\in (0,T)}|K(t,M)|=\infty$ is rarely studied.  Another kind of the source was studied in \cite{THNZ2018, TNHK2018}, but only for the backward problem. In the present paper, we consider the singular source with $K(t,M)=t^{-\nu} \kappa(M)$, $\nu>0$. The source is similar to  the one considered in a paper of Nagumo (see, e.g., \cite{Gorenflo}, Chap. 7). As known (see, e.g., \cite{Fabry}), the function $h:(0,\infty)\to (0,\infty)$ satisfies the Nagumo condition if $\int_0^\infty \frac{s}{h(s)}ds=\infty$. Hence,  the function $K$ is a Nagumo function and we call the source $f(t,u)$ is a Nagumo-type source.
The existence and uniqueness of maximal solution of the initial value  problem with respect to the singular source is still not studied.
{\it This is the first motivation of our paper.}

In the  papers mentioned above, the parameters $\alpha,\beta$ are assumed to be perfectly known.
But in the real word of applications, the fractional orders can only be approximated from the mathematical model
or statistical methods. In \cite{Aldoghaither-siam, cheng2009uniqueness}, the Caputo derivatives
can be identified approximately from observation data $u(x_0,t)$ with $t>0$, or $u(x,T)$ with $x\in\Omega\subset \mathbb{R}^n$.
Besides, Kateregga \cite{Kateregga2017}
used statistical methods as the quantiles, logarithmic moments method, maximum likelihood,
and the empirical characteristic function method
to identify the parameters of the L\'{e}vy process.
In these examples, the fractional orders are obtained only as approximate values.
Hence, a natural question is that whether the solutions of fractional equations is continuous
with respect to the perturbed orders.
The papers devoted to these questions are still rare.
We can list here some papers. Li and Yamamoto \cite{LiZhangJiaYamamoto2013}
investigated the solution $u_{\gamma, D}$ of the problem
\[
  D_t^\gamma u
 =
  \frac{\partial}{ \partial x} \left(D(x)\frac{\partial u}{\partial x} \right),
  (x, t) \in (0, 1) \times(0, T),
\]
subject to the Neumann condition $u_x(0, t) =u_x(1, t)=0$ and the initial condition $u(x, 0)=f(x)$. They proved
\[
   \| u_{\gamma_1, D_1} - u_{\gamma_2, D_2}\|_{L^2(0, T)}
\le
   C(|\gamma_1 -\gamma_2| + \|D_1 - D_2\|_{C[0,1]}).
\]
Trong et al \cite{trong2017potential} studied the continuity of solutions
of some linear fractional PDEs with perturbed orders.
In \cite{dt2019,tdv2019, vdt2019}, we investigated stability of solution
of some class of nonlinear space--fractional diffusion problems taking into account the disturbance of parameters.
In our knowledge, until now, we do not find another paper which considers the stability
of the nonlinear  FIVP with respect to the parameters $\alpha, \beta$.
{\it This is the second motivation of our paper.}

Besides the FIVP, we also consider the FFVP.
For the classical problem with $\alpha=1$, it is well-known that this problem is ill-posed in the sense of Hadamard.
The proposed methods to regularize these kind of problems are very abundant such as:
quasi--reversibility  \cite{JN2016}, quasi--boundary value \cite{DB2001, ttq2009},
Tikhonov \cite{ZM2010}, truncated \cite{NTT2010, trong2011regularization},
Landweber iteration and iterative Lavrentiev \cite{LX2012}, etc.

Recently, the  FFVP with $\alpha\in (0,1)$ was investigated.
Different from the case $\alpha=1$,  the linear FFVP is stable for $0<t<T$ and instable at $t=0$ (see \cite{wei2013, Tuan2017}).
Hence, only regularization of solution at $t=0$ is needed.
For the linear case, Ting Wei et al \cite{wei2014qr} used the Tikhonov method to regularizing the homogeneous problem.
Tuan et al \cite{Tuan2017} also used the Tikhonov method to regularize the nonhomogeneous
time--fractional problem
\begin{equation}
 \label{inhomo}
    D_t^\alpha u
   =
    L u + f(x, t), \  u(x,T)=u_T(x),
\end{equation}
where
$
   L u
  =
  \sum_{i=1}^d
   \frac{\partial}{\partial x_i}\left(A_{ij}\frac{\partial u}{\partial x_j} \right)
   +C(x) u
$.
Yang et al \cite{YangRenLi} also regularize the problem (\ref{inhomo}) by the quasi-reversibility method.
For the nonlinear case, the problem is completely different since the integral form of the problem is a nonlinear Fredholm Integral equation. Unlike the fractional linear case and the classical nonlinear final value heat equation,
the nonlinear FFVP is nonlocal and can not be transformed into a Volterra-type equation. Some pioneering results on existence and uniqueness solutions of some FFVP
were studied in \cite{THNZ2018, TNHK2018}.  The papers considered the operator $A$ having positive discrete eigenvalues and gives uniqueness and existence results of solution of FFVP for $T$ small.

From the  overview above, we discuss the motivation of the paper for the FFVP. The stability of
the solution of nonlinear FFVP with respect to unknown fractional orders $\alpha,\beta$  is still not investigated in the mentioned paper. Therefore, in the current paper, we study the stability of the nonlinear FFVP on the special space $C_{s,T}$ (defined in Section 2).

We note that the norm of the space cannot give any information of its functions at $t=0$.
Hence, we have to consider separately the case $t=0$. As mentioned, the problem is instability at $t=0$ and
we cannot find any paper that  dealt with the instability
and regularization for the nonlinear case with perturbed fractional orders has to be established. {\it This is the third motivation  of the paper.}

\subsection{Outline of the paper}

Summarizing  the discussion of the FIVP and the FFVP, in the present papers, we will

\begin{itemize}

\item[$\bullet$]
Investigate the existence and uniqueness of the maximal solution of the nonlinear FIVP with respect to a singular nonlinear source. To solve the problem, we have to establish an appropriate Gronwall-type inequality which also has a specific  merit  in investigating other fractional problems.

\item[$\bullet$]
Study the stability of the nonlinear FIVP with respect to  the perturbed orders.
For $\alpha \to 1^-$, we will prove that  the solution of the nonlinear FIVP tends to  that of classical
nonlinear initial value parabolic problem.

\item[$\bullet$]
Consider the existence and uniqueness solution of the FFVP for $t>0$.
In the case the problem has a unique solution, some stability results will be given.

\item[$\bullet$]
Analyze the ill-posedness of the FFVP with respect to the unknown fractional parameters $\alpha,\beta$ at $t=0$,
and establish a  method to regularize the solution of the FFVP.
\end{itemize}

The rest of the papers is divided into four sections.
The second section is devoted to some notations, definitions, and properties of the Mittag-Leffler function.
In the third section, we consider the existence, uniqueness, and stability of solutions of the nonlinear FIVP with respect to the perturbed orders.
In the fourth section, we investigate the existence, uniqueness, and stability of solution of the FFVP for $t>0$
and propose a method to regularize the FFVP at $t=0$.
In Sections 3-4 we only present the proofs of main theorems. Proofs of Lemmas will be given in the fifth section.

\section{Preliminary results}
To state precisely our problem, we will give some definitions. We denote the inner product in Hilbert space $H$ by $\langle. ,.\rangle$ and the associated norm by $\|.\|$.
For $\rho, \ct>0$, we define
\begin{equation*}
   \cc^\rho(\ct)
=
   \left\{\psi \in C\left((0, \ct] \right) :\ \int_0^\ct(\ct-\tau)^{\rho-1} | \psi(\tau) |^2 \ud \tau <\infty \right\}
.
\end{equation*}

Let us denote by $\{\es\}$ the spectral resolution of the identity associated to operator $A$.
We follow \cite[page 29]{yakubov} to define the power of the self-adjoint positive definite unbounded operator as
\[
  A^\beta u = \int_\theta^{+\infty} \lambda^\beta \ud \es u, \ \ \beta \in \mathbb{R},
\]
where $\theta$ is the lower bound of the spectrum of the operator $A$.

Generally, for a continuous function  $h: \mathbb{R} \to \mathbb{R}$, we denote the domain of $h(A)$ to be
\begin{equation}
\label{spectral-representation}
D(h(A))
:=
\left\{
      w\in H:
      \
     \int_\theta^{+\infty} |h(\lambda)|^2 \ud \dnorm{\es w}^2<+\infty
\right\}.
\end{equation}
If $w \in D(h(A))$, we define the linear operator
\[
   h(A) w =\int_\theta^{+\infty} h(\lambda) \ud \es w.
\]
In particularly, if $h(A)=A^s$ for some $s \ge 0$, we have the Hilbert space $D(A^s)$ with the norm
$
\dnorm{w}_s
=
\left( \int_\theta^{+\infty} \lambda^{2 s} \ud \dnorm{\es w}^2 \right)^{1/2}
$.
Let $0 \le s_* \le s^*$
and
$s_1, s_2  \in [s_*, s^*]$, $s_2\leq s_1$.
It is easy to see that
\[
D(A^{s_1})
\subset
D(A^{s_2})
\subset
D(A^0) = H
\, \,
\text{and}
\,\,
\dnorm{w}_{s_2}
\le
\theta^{s_2-s_1} \dnorm{w}_{s_1}
.
\]

For $\m, s>0$, we denote
 the closed ball of radius $\m>0$ centered at origin
in the Banach spaces $D(A^s)$
and
$C( [0,T] ; D(A^s ))$ by
\begin{eqnarray*}
B_s(\m)
&=&
\{
v\in D(A^s ): \|v\|_s\leq\m
\},
\nn
\\
B_{s,T}(\m)
&=&
\left\{
v\in C([0,T];D(A^s )) : | v |_{s,T} \leq \m
\right\},
\end{eqnarray*}
respectively, where $|v|_{s,T}=\sup_{0 \le t \le T} \| v(t)\|_s$.

For $\mathcal{T},\rho, s>0$, we denote
\[
C_{s,\rho}(\mathcal{T})
=
\{w\in C((0,\mathcal{T}], D(A^s)):\ \tnorm{w}_{s, \rho}<\infty\}
\]
where
$
   \tnorm{w}_{s, \rho} =\sup_{t \in (0, T]} t^\rho \dnorm{w(t)}_s
   .
$

Given the  notations defined above, we can state precisely the assumption for the singular source.
In fact, we consider the source function of the problem satisfying the following assumption.
\\

\noindent{\bf Assumption F1}.
{\it
Let $\beta>0$,  $s \in [0, \beta/2]$, $\nu \le \alpha/2$
and
$f: (0,T] \times D(A^s) \to L^2(0,T; H)$.
For every  $\m>0$, we
assume that
\begin{equation}
   L(s,\m):=   \sup_{0<t\leq T}\sup_{(w_1,w_2) \in \mathcal{D}_s (\m)}\frac{t^\nu \dnorm{f(t,w_1)-f(t,w_2)}}
 {\dnorm{ w_1-w_2 }_s}<\infty.
\label{local-lips}
\end{equation}
where
$\mathcal{D}_s(\m)
=
\{(w_1,w_2):\ w_1,w_2\in D(A^s), \| w_1\|_s, \| w_2\|_s\leq\m,  w_1\not= w_2\}$.
}

Put $\Omega=(0, 1)$, $H=L^2(\Omega)$,
we can directly check that some common sources of the following equations
satisfy  \textbf{Assumption F1} :
the Ginzburg Landau equation for $\nu=0, s=0$,
the Burger equation for $\nu=0, s=1$,
the Cahn--Hilliard and Kuramoto--Sivashinsky equations for $\nu=0, s=2$.

To investigate the stability of the solution of the  two problems (FIVP, FFVP),
we will restrict the value $(\alpha, \beta)$ in the bounded domain $\Delta$.
More precisely, for $0<\alpha_*<\alpha^*<2$, $\alpha^*<2\alpha_*$, and $0<\beta_*<\beta^*$,
the domain $\Delta$ is defined as
\begin{equation}
\label{delta}
   \Delta
   =
   \left\{
      (\alpha, \beta): \ \alpha_* \le \alpha \le \alpha^*,
      \beta_*\le \beta \le \beta^*
   \right\}
   .
\end{equation}

In this section,  we also introduce the Mittag--Leffler function and its properties
which play important roles in the proof of main results of current paper. We recall that the  Gamma and Beta functions  are
$$ \Gamma(z)=\int_0^\infty t^{z-1}e^{-t}dt,\ B(p,r)=\int_0^1t^{p-1}(1-t)^{r-1}dt\ \ \ \text{for}\ \text{Re}(z), p, r>0. $$
The Mittag--Leffler function with two parameters is defined as
   \begin{equation*}
     \fm_{p, r} (z)
  =
     \sum_{k=0}^{+\infty} \frac{z^k}{\Gamma(k p+r)}, \fm_{p} (z):=\fm_{p, 1} (z),
     \,\, z \in \mathbb{C} \ \ \ \text{for}\ p, r  >0.
   \end{equation*}




\begin{lemma}[see \cite{sakamoto2011fractional}]
\label{sakamoto}
   Letting $\lambda>0, p>0$ and $k\in \mathbb{N}$, we have
   \[
     \frac{\ud^k}{\ud t^k} \fm_{p}(-\lambda t^p)
  =
     - \lambda t^{p - k } \fm_{p, p-k+1}(-\lambda t^p), \, \, t\ge 0.
   \]
 \end{lemma}

\begin{lemma}
\label{trong1}
Let  $0<p_*<p^*<2$ such that $p^*<2p_*$, and $r_*>0$.
Then for any $p, p_0 \in [p_*, p^*]$, and $r, r_0 \ge r_*$, and $\lambda \ge 0$,
we have

(a).
There exists a constant $C=C(p_*, p^*, r_*)>0$ such that
\[
\snorm{\fm_{p, r}(-\lambda)}
+
\snorm{\frac{\partial \fm_{p, r}}{\partial p}(-\lambda)}
+
\snorm{\frac{\partial \fm_{p, r}}{\partial r}(-\lambda)}
\le
C.
\]
   Moreover, we have
\begin{equation}
   0\leq E_\alpha(-z)\leq 1,\ 0\leq E_{\alpha,\alpha}(-z)\leq \frac{1}{\Gamma(\alpha)}\ \ \ \text{for}\ z\geq 0.
\label{Mittag-bound}
\end{equation}
(b).
Let $0<p_*<p^*<1$.
There exist two constants $C_1, C_2$ which  depend only  on $p_*, p^*$
such that
\[
\frac{1}{\Gamma(1-p)}
\frac{C_1}{1+\lambda}
\le
\fm_{p}(- \lambda)
\le
\frac{1}{\Gamma(1-p)}
\frac{C_2}{1+\lambda}
.
\]

(c).
There exists a constant $C=C(p_*, p^*)$ such that
\[
\left|
  \fm_{p}(-\lambda^r t^p)
  -
  \fm_{p_0}(-\lambda^{r_0} t^{p_0})
\right|
\le
C \lambda^{r^*} (1+\ln \lambda)
 \left(|p - p_0| +|r - r_0 | \right),
\, \,
 \forall \, \lambda \ge 1,
\]

(d).
We denote
\[
\E_{a, b}(\lambda, t, \tau)
=
(t-\tau)^{a-1} \fm_{a, a}(-\lambda^b (t-\tau)^a).
\]
Then, there exists a constant $C=C(p_*, p^*, r_*)$ such that
\begin{equation*}
\int_0^t
\left|
   \E_{p, r}(\lambda, t, \tau) - \E_{p_0, r_0}(\lambda, t, \tau)
\right|
\ud \tau
\le
C
\left((1+\lambda^r)|p- p_0| +|\lambda^r - \lambda^{r_0} | \right).
\end{equation*}


\end{lemma}
\begin{proof}
 We only prove \eqref{Mittag-bound}. Readers can see the proof of other cases in \cite{trong2017potential}.
From the complete monotonicity of the Mittag-Leffler function $E_{\alpha}(-z)$ for $z\geq 0$ (see \cite{Gorenflo-Mittag}, Chap. 3) we have $(-1)^n\frac{d^n}{dz^n}E_\alpha(-z)\geq 0$ for $z\geq 0$.
Hence we have  $E_\alpha(-z), E_{\alpha,\alpha}(-z)$ is decreasing which give
 $0\leq E_{\alpha}(-z)\leq 1$, $0\leq E_{\alpha,\alpha}(-z)\leq \frac{1}{\Gamma(\alpha)}$ for $z\geq 0$.
\end{proof}
In this paper, we also need the useful inequality

\begin{lemma}
\label{gronwall}
~
Let $\alpha, q\in \mathbb{R}$, $0<\alpha\leq 1$, $q<\alpha$, and let $v, g\in C[0, T]$.
Then the equation
\[
  u(t)
=
   v(t) + g(t) \int_0^t (t - \tau )^{\alpha -1} \tau^{-q} u(\tau) \ud \tau
\]
has a unique solution $u\in C[0,T]$ which  satisfies
\begin{equation}
   |u(t)|
 \le
   \Gamma(1-q) \| v\|_{C[0,t]}
   E_{\alpha-q, 1-q} \left(\| g\|_{C[0,t]}\Gamma(\alpha) t^{\alpha-q}\right).
\label{u-bound}
\end{equation}
for  $t \in [0,T]$.
As a consequence, if $w\in C[0,T]$ satisfies
\[
   0
\le
   w(t)
\le
   v(t)+g(t)\int_0^t (t-\tau)^{\alpha-1}\tau^{-q}w(\tau)d\tau
   \ \ \ \ \text{for}\ t\in [0,T],
\]
and if $g(t)\geq 0$ for $t\in [0,T]$, then
\[
    w(t)
 \le
   \Gamma(1-q) \| v\|_{C[0,t]}
   E_{\alpha-q, 1-q} \left(\| g\|_{C[0,t]}\Gamma(\alpha) t^{\alpha-q}\right)
    \ \ \text{for}\ t\in [0,T].
\]
\end{lemma}

\begin{proof}
See Appendix \ref{gronwall-app}.
\end{proof}

We also need the  following results

\begin{lemma}
\label{Q-lemma}
Let $T, \theta>0, \alpha \in (0, 1], \beta>0$, $s \in [0, \beta/2]$, $r \ge 0$,  $t_1,t_2\in (0,T]$, $t_1<t_2$,
 $w\in C([0,T]; D(A^r))$.

(i) For $\zeta\in D(A^{s+r})$, we have
$ E_\alpha(-t^\alpha A^\beta)\zeta\in  D(A^{s+r})$ and
$\| E_\alpha(-t^\alpha A^\beta)\zeta\|_{s+r}\leq \|\zeta\|_{s+r} $.

(ii)
Put
\begin{eqnarray}
\label{Q-define}
Q_{ \alpha, \beta, A}(w)(t_1,t_2)
 &=&
\int_{t_1}^{t_2}
   \E_{\alpha, \beta}(A, t_2, \tau) w(\tau)
\ud \tau,
\end{eqnarray}
where
\[
   \E_{\alpha, \beta}(\lambda, t, \tau)
=
   (t-\tau)^{\alpha-1} \fm_{\alpha, \alpha}(-\lambda^\beta (t-\tau)^\alpha)
=
   \frac{1}{\lambda^\beta}\frac{d}{d\tau}E_\alpha(-\lambda^\beta(t-\tau)^\alpha).
\]
If $w\not\equiv 0$ on $[t_1,t_2]$ then
\begin{equation}
\label{Q-es}
   \dnorm{Q_{ \alpha, \beta, A}(w)( t_1,t_2 )}_{s+r}^2
   <
   \frac{1}{\Gamma(\alpha)}\
   \sup_{\lambda\geq\theta} \lambda^{2s-\beta}H_0(\lambda,t_1,t_2)
   \int_{t_1}^{t_2} (t - \tau )^{\alpha - 1} \dnorm{ w(\tau) }_r^2  \ud \tau.
\end{equation}
where
\[
   H_0(\lambda,t_1,t_2)
 :=
   \left(1- E_{\alpha}\big(- \lambda^\beta (t_2-t_1)^\alpha \big)\right).
\]

(iii) Put
\begin{eqnarray}
  R_{ \alpha, \beta, A}(w)( t_1, t_2 )
&=&
  \int_0^{t_1}
     |\E_{\alpha, \beta}(A, t_1, \tau)-\E_{\alpha, \beta}(A, t_2, \tau)| w(\tau)
  \ud \tau
 .
 \nn
\end{eqnarray}
Then
\begin{eqnarray}
   \dnorm{R_{ \alpha, \beta, A}(w)( t_1, t_2 )}_{s+r}^2
&\le &
   \frac{1}{\Gamma(\alpha)}\
   \sup_{\lambda\geq\theta}\lambda^{2s-\beta}H(\lambda,t_1,t_2)
   \int_0^{t_1}
   (t_1 -\tau)^{\alpha-1}  \dnorm{w(\tau)}_r^2
   \ud \tau.
\label{equicontinuous-estimate-1}
\end{eqnarray}
where
\[
   H(\lambda,t_1,t_2)
=
   1-E_\alpha(-\lambda^\beta (t_2-t_1)^\alpha)
+
   E_\alpha(-\lambda^\beta t_2^\alpha)-E_\alpha(-\lambda^\beta t_1^\alpha).
\]

\end{lemma}

\begin{proof}
See appendix \ref{Q-lemma-app}.
\end{proof}

\section{The FIVP}

The continuity of solution of the homogeneous FIVP studied in \cite{trong2017potential},
but  not studied for the nonlinear case.
Hence in this section, we study the well-posedness of solution of the FIVP with nonlinear source.
In fact, we will prove  existence results for  the mild solutions of the problem.
In the case the  problem has a unique mild solution,
we prove that this solution  depends continuously on the fractional orders $\alpha, \beta$ and
the initial data $\zeta$.

By the definition of the spectral resolution of the operator $A$ and the Laplace transform, we can rewrite the FIVP as
the following integral equation

\begin{equation} \label{mild-solution}
u(t)
 =
     \fm_\alpha \big(- t^\alpha A^\beta \big) \zeta
    +
    \int_0^t
       \E_{\alpha, \beta}(A, t, \tau)
       f(\tau, u(\tau))
    \ud \tau,
\end{equation}
where
$
\E_{\alpha, \beta}(z, t, \tau)
=
(t-\tau)^{\alpha-1} \fm_{\alpha, \alpha}(-z^\beta (t-\tau)^\alpha)
$.
A function $u$ that satisfies  Eq. (\ref{mild-solution})
is called a mild solution of the FIVP.
We prove that the FIVP has a unique mild solution that
depends continuously on the input data: fractional order and the initial value data.

Firstly, we have the following existence result.

\begin{theorem}
\label{existence-thm}
Let $0<\alpha\leq 1$, $\beta>0$,  $\nu \le \alpha/2$, $0 \le s \le \beta/2$
and
$\zeta \in D(A^s)$.
Let $f \in C\left((0,T)\times D(A^s); H\right)$ and let $h: [0,T]\to\mathbb{R}$ be Lebesgue measurable.
Suppose that there exists a increasing function $\psi: [0, +\infty) \to [0, +\infty)$ such that
\begin{equation}
    \dnorm{f(t,w)}
\le
    \kappa t^{-\nu} \psi(\dnorm{w}_s) +h(t).
\label{F2}
\end{equation}
If
\[
m_T
:=
\sup_{0 \le t \le T} \int_0^t (t-\tau)^{\alpha-1} |h(\tau)|^2 \ud \tau
<\infty
\]
and there exists $m>0$ such that
\begin{equation}
\label{m-inq}
m
>
\|\zeta\|_s
+
\left(\frac{2}{\Gamma(\alpha)} \right)^{1/2} \theta^{s-\beta/2}
\left(
   m_T
  +
   B(\alpha, 1-2\nu) \kappa^2 T^{\alpha-2\nu}
   (\psi(m))^2
\right)^{1/2}
,
\end{equation}
then the problem (\ref{mild-solution}) has at least one solution $u\in C([0,T];D(A^s))$.
\end{theorem}
\begin{remark}
Let $n \in \mathbb{N}$, $a_k \geq 0$ and $p_k \in [0, 1)$ for $k=1, 2,.., n$. The condition \eqref{m-inq} holds if
$
   \psi(z)=\sum_{k=1}^n a_k z^{p_k}
$
for  $z\geq 0$.

\end{remark}
\begin{proof}
Let us  denote
\[
   \F(w)(t)
 =
     \fm_\alpha \big(- t^\alpha A^\beta \big) \zeta
    +
    \int_0^t
       \E_{\alpha, \beta}(A, t, \tau)
       f(\tau, w(\tau))
    \ud \tau.
\]
for $w\in C([0,T],D(A^s))$,
and put
\[
\Omega=\left\{ w \in C([0,T],D(A^s)): | w |_{s,T}<m \right\}.
\]

We show that $\F$ is completely continuous.
Indeed, since $f$ is a continuous function, we can easily  show that $\F$ is continuous.
Using Lemma \ref{Q-lemma} with $r=0, t_1=0, t_2=t$, we also have
\begin{eqnarray}
\label{bound-es}
   \dnorm{\F (w)(t)}_s
&\le &
   \dnorm{\zeta}_s
   +
    \left(\frac{1}{\Gamma(\alpha)}\right)^{1/2} \theta^{s-\beta/2}
    \left(
    \int_0^t
       (t-\tau)^{\alpha-1} \dnorm{ f(\tau, w(\tau))}^2
    \ud \tau.
    \right)^{1/2}
\nn
\\
&\le &
   \dnorm{\zeta}_s
   +
    \left(\frac{2}{\Gamma(\alpha)}\right)^{1/2} \theta^{s-\beta/2}
    \left(
    \int_0^t
       (t-\tau)^{\alpha-1}
       \left(
          \kappa^2 \tau^{-2\nu} \left(\psi(\dnorm{w(\tau)}_s)\right)^2 +\snorm{ h(\tau)}^2
       \right)
    \ud \tau
    \right)^{1/2}
\nn
\\
&\le &
\|\zeta\|_s
+
\left(\frac{2}{\Gamma(\alpha)}\right)^{1/2} \theta^{s-\beta/2}
\left(
   m_T
  +
   B(\alpha, 1-2\nu) \kappa^2 T^{\alpha-2\nu} (\psi(m))^2
\right)^{1/2}
\end{eqnarray}
for any $u \in \Omega$.
Hence $\F(\Omega)$ is bounded.
Moreover, for $t_1, t_2 \in [0, T], \ t_1<t_2$ and $w \in \Omega$, we have
\begin{eqnarray*}
   \dnorm{\mathcal{F}(w)(t_1)-\mathcal{F}(w)(t_2)}_s
& \le &
   \dnorm{
      \left(\fm_\alpha \big(- t_1^\alpha A^\beta \big)-\fm_\alpha \big(- t_2^\alpha A^\beta \big)\right)
      \zeta
   }_s
\\
&+&
   \dnorm{
   R_{\alpha,\beta,A}(f(.,w))(t_1,t_2)}_s
+
   \dnorm{ Q_{\alpha,\beta,A}(f(.,w))(t_1,t_2)}_s
\end{eqnarray*}
where $ Q_{\alpha,\beta,A}(f(.,w))(t_1,t_2), R_{\alpha,\beta,A}(f(.,w))(t_1,t_2)$ are defined in Lemma \ref{Q-lemma}.
For $2s \le \beta$, $r=0$, the quantities
$ \sup_{\lambda\geq \theta}H(\lambda,t_1,t_2), \sup_{\lambda\geq \theta}H_0(\lambda,t_1,t_2)$
 in Lemma \ref{Q-lemma} satisfy
\[
 \lim_{\delta\to 0}\sup_{|t_1-t_2| \le \delta}
 \sup_{\lambda\ge \theta} H_0(\lambda,t_1,t_2)
 =
\lim_{\delta\to 0}\sup_{|t_1-t_2|\le\delta}
\sup_{\lambda\ge \theta}H(\lambda,t_1,t_2)
=
0.
\]
Hence, we can use Lemma \ref{Q-lemma} to verify directly the set $\F(\Omega)$ is equicontinuous.

This shows that $\F: \overline{\Omega} \to C([0,T],D(A^s))$  is completely continuous.
We suppose that there exists $u \in \partial \Omega$ and $\mu \in (0, 1)$ such that
$
   u= \mu \F u.
$
We can use (\ref{bound-es}) to get the estimate
\begin{eqnarray*}
   \dnorm{u(t)}_s
&=&
   \mu\dnorm{\F(u)(t)}_s
\\
&\le &
\|\zeta\|_s
+
\left(\frac{2}{\Gamma(\alpha)}\right)^{1/2} \theta^{s-\beta/2}
\left(
   m_T
  +
   B(\alpha, 1-2\nu) \kappa^2 T^{\alpha-2\nu}
   (\psi(m))^2
\right)^{1/2}
\end{eqnarray*}
or
\[
   m
\le
   \|\zeta\|_s
+
   \left(\frac{2}{\Gamma(\alpha)}\right)^{1/2} \theta^{s-\beta/2}
   \left(
      m_T
    +
      B(\alpha, 1-2\nu) \kappa^2 T^{\alpha-2\nu}
      (\psi(m))^2
   \right)^{1/2}
.
\]
This  contradicts with (\ref{m-inq}).
Hence, the nonlinear Leray-Schauder alternatives fixed point theorem (see \cite[p.4]{Granas})
implies that
$\F$ has a fixed point $u\in \overline{\Omega}$
or the problem (\ref{mild-solution}) has a solution in $\overline{\Omega}$.
This completes the proof of the Theorem.
\end{proof}

\begin{theorem}
\label{global-existence}
Let $\alpha\in (0,1),\beta>0$,  $s \in [0, \beta/2]$,
and let
$f: (0,T)\times D(A^s)\to L^2(0,T; H)$.
Assume that the condition \eqref{local-lips} holds and
\[
    \kappa
 =
    \sup_{\m>0}L(s,\m)<\infty.
\]

For $\nu<\alpha/2$, the equation \eqref{mild-solution} has a unique solution $u\in C([0,T];D(A^s))$.
Moreover, if $g \in C\left([0, T], D(A^s)\right)$ with
\begin{equation}
\label{g-condition}
    g(t)
 =
     \fm_\alpha \big(- t^\alpha A^\beta \big) \zeta
    +
    \int_0^t
       \E_{\alpha, \beta}(A, t, \tau)
       f(\tau, 0)
    \ud \tau
,
\end{equation}
then
\begin{equation}
\label{global-bound}
   \| u(t) \|_s^2
\le
   2 \Gamma(1-2\nu) \| g\|_{s, t}^2
   E_{\alpha-2\nu,1-2\nu} \left(2\theta^{2s-\beta}\kappa^2 t^{\alpha-2\nu}\right)
\end{equation}
for any $t\in [0,T]$.

For
$\nu=
\alpha/2$,
$\kappa < \theta^{\beta/2-s} \left( \Gamma(1-\alpha)\right)^{-1/2}$,
the equation \eqref{mild-solution} has a unique solution $u\in C([0,T];D(A^s))$.

 For
 $\nu = \alpha/2$,
 $\kappa = \theta^{\beta/2-s} \left( \Gamma(1-\alpha)\right)^{-1/2}$, $0\leq s<\beta/2
 $,
 if we have in addition that the assumptions
\eqref{F2} and \eqref{m-inq} hold
then the equation \eqref{mild-solution} has a unique solution $u\in C([0,T];D(A^s))$.
\end{theorem}
\begin{proof}
For $w\in C([0,T],D(A^s))$, we put
\[
   \F(w)(t)
 =
     \fm_\alpha \big(- t^\alpha A^\beta \big) \zeta
    +
    \int_0^t
       \E_{\alpha, \beta}(A, t, \tau)
       f(\tau, w(\tau))
    \ud \tau.
\]
Choosing $t_1=0, t_2=t, r=0$  in Lemma \ref{Q-lemma} gives
$\sup_{\lambda\geq\theta}\lambda^{2s-\beta}H_0(\lambda,t_1,t_2)\leq \theta^{2s-\beta}$. Hence, we obtain in view of Lemma \ref{Q-lemma}
\begin{eqnarray}
  \dnorm{\F(w_1)(t)-\F(w_2)(t)}^2_s
&\le &
   \frac{1}{\Gamma(\alpha)}
   \theta^{2s-\beta}
   \int_0^t
      (t -\tau)^{\alpha-1}  \dnorm{f(\tau, w_1(\tau))-f(\tau, w_2(\tau)}^2
   \ud \tau
\nn
\\
&\le &
   \frac{1}{\Gamma(\alpha)}
   \theta^{2s-\beta}
   \kappa^2\int_0^t
      (t -\tau)^{\alpha-1} \tau^{-2\nu} \dnorm{w_1(\tau)-w_2(\tau)}_s^2
   \ud \tau.
\label{T-Lipschitz}
\end{eqnarray}
So we have
\begin{eqnarray*}
    \dnorm{\F(w_1)(t)-\F(w_2)(t)}^2_s
&\le &
    \frac{1}{\Gamma(\alpha)}
    \theta^{2s-\beta} \kappa^2 |w_1-w_2|_{s,T}^2
    \int_0^t(t-\tau)^{\alpha-1}\tau^{-2\nu}d\tau
 \\
&=&
   \frac{1}{\Gamma(\alpha)} B(\alpha,1-2\nu)
   \theta^{2s-\beta} \kappa^2
   |w_1-w_2|_{s,T}^2 t^{\alpha-2\nu}
\\
&=&
   \frac{\Gamma{(1-2\nu)}}{\Gamma(\alpha+1-2\nu)}
   \theta^{2s-\beta}\kappa^2|w_1-w_2|_{s,T}^2 t^{\alpha-2\nu}.
\end{eqnarray*}
We consider the case $\nu<\alpha/2$. For $w_1,w_2\in C([0,T],D(A^s))$, using the similar technique as in \cite{trong2017potential},
we can prove by induction that
\[
    \dnorm{\F^k(w_1)(t)-\F^k(w_2)(t)}^2_s
\le
    \frac{
       \Gamma(1-2\nu)
       \left(\theta^{2s-\beta} \kappa^2\right)^k
       t^{k(\alpha-2\nu)}
       }{
       \Gamma(k(\alpha-2\nu)-2\nu+1)
       }
       |w_1-w_2|^2_{s,T}
       .
\]
We note that
\[
   \lim_{k\to\infty}
    \frac{
          \Gamma(1-2\nu)
          \left(\theta^{2s-\beta} \kappa^2\right)^k
          T^{k(\alpha-2\nu)}
       }{
          \Gamma\left(k(\alpha-2\nu)-2\nu+1\right)
       }
=0.
\]
Hence there is a $k_0\in\mathbb{N}$ such that
\[
    \frac{
          \Gamma(1-2\nu)
          \left(\theta^{2s-\beta} \kappa^2\right)^{k_0}
          T^{k_0(\alpha-2\nu)}
       }{
          \Gamma\left(k_0(\alpha-2\nu)-2\nu+1\right)
       }
\le
   \frac{1}{2}
\]
which gives
\[
    | \F^{k_0}(w_1)-\F^{k_0}(w_2)|^2_{s,T}
\le
    \frac{
          \Gamma(1-2\nu)
          \left(\theta^{2s-\beta} \kappa^2\right)^{k_0}
          T^{k_0(\alpha-2\nu)}
       }{
          \Gamma\left(k_0(\alpha-2\nu)-2\nu+1\right)
       }
          |w_1-w_2|^2_{s,T}
\le
   \frac{1}{2}
   |w_1-w_2|^2_{s,T}
,
\]
i.e., $\F^{k_0}$ is a contraction in $C([0,T],D(A^s))$. Hence, the exists a unique fixed point $u\in C([0,T],D(A^s)) $
satisfying $u=\F^{k_0}(u)$. We deduce that $\F u=\F^{k_0}(\F u)$, i.e., $\F u$ is also a fixed point of the operator $\F^{k_0}$. Hence $u=\F u$.

We give the estimate of $u$. In fact, from \eqref{T-Lipschitz} we obtain
\[
   \| u(t)-g \|^2_s
=
   \dnorm{\mathcal{F}u(t)-\mathcal{F}(0)(t)}^2_s
\le
   \frac{1}{\Gamma(\alpha)} \theta^{2s-\beta}\kappa^2
   \int_0^t
      (t -\tau)^{\alpha-1} \tau^{-2\nu} \dnorm{u(\tau)-0}_s^2
   \ud \tau.
\]
Hence
\[
    \| u(t) \|_s^2
\le
    2\| g(t) \|_s^2+2\| u(t)-g\|^2_s
\le
   2\| g(t)\|^2_s+\frac{2}{\Gamma(\alpha)}\theta^{2s-\beta}\kappa^2
   \int_0^t
      (t -\tau)^{\alpha-1} \tau^{-2 \nu} \dnorm{u(\tau)}_s^2
   \ud \tau.
\]
Using \eqref{u-bound} of Lemma \ref{gronwall}, we obtain the inequality of the Theorem.

Finally, we consider the case $\nu=\alpha/2$. We can find a $\xi\in (0,T]$ such that
$  \dnorm{\F(w_1)(\xi)-\F(w_2)(\xi)}^2_s=\sup_{0\leq t\leq T}  \dnorm{\F(w_1)(t)-\F(w_2)(t)}^2_s$.
Lemma \ref{Q-lemma} gives
\begin{eqnarray*}
  \dnorm{\F(w_1)(\xi)-\F(w_2)(\xi)}^2_s
&< &
   \frac{1}{\Gamma(\alpha)}\theta^{2s-\beta}
   \int_0^\xi
      (\xi -\tau)^{\alpha-1}  \dnorm{f(\tau, w_1(\tau))-f(\tau, w_2(\tau)}^2
   \ud \tau
\nn
\\
&\le &
   \frac{1}{\Gamma(\alpha)} \theta^{2s-\beta}
   \kappa^2
   \int_0^\xi
      (\xi -\tau)^{\alpha-1} \tau^{-\alpha} \dnorm{w_1(\tau)-w_2(\tau)}_s^2
   \ud \tau
\\
&\le &
   \frac{1}{\Gamma(\alpha)} B(\alpha, 1-\alpha) \theta^{2s-\beta}\kappa^2
   |w_1-w_2|_{s,T}^2
\\
&=&
   \Gamma(1-\alpha) \theta^{2s-\beta}\kappa^2
   |w_1-w_2|_{s,T}^2.
\end{eqnarray*}
It follows
\begin{equation}
|\mathcal{F}(w_1)-\mathcal{F}(w_2)|_{s,T}<|w_1-w_2|_{s,T}.
\label{T0-Lipschitz}
\end{equation}
~
If $\kappa < \theta^{\beta/2-s} \left( \Gamma(1-\alpha)\right)^{-1/2}$ then $\F$ is a contraction in $C([0,T],D(A^s))$.
Consequently, the problem (\ref{mild-solution}) has a unique solution in $C([0,T],D(A^s))$.
~
\\
Finally, we consider the case $\kappa=\theta^{\beta/2-s} \left( \Gamma(1-\alpha)\right)^{-1/2}$.
From Theorem \ref{existence-thm}, the problem \eqref{mild-solution} has a solution
$u\in C([0,T],D(A^s))$. From  the inequality \eqref{T0-Lipschitz} we deduce
that the solution is unique.
\end{proof}
\begin{remark}
We can use the Edelstein fixed point theorem (see, e.g., \cite{Gorenflo}, Chap. 7) to obtain the desired result for the case $\nu=\alpha/2, \kappa=\theta^{\beta/2-s} \left( \Gamma(1-\alpha)\right)^{-1/2}$. We note that if we put $u_{n+1}=\mathcal{F}(u_n)$ then the sequence $(u_n)$ converges to the solution u in $C([0,T],D(A^s))$.
\end{remark}

Now, we investigate the existence and uniqueness of the solution of the problem with local source defined in (\ref{local-lips}).
In addition, we study the dependence of the solution with respect to the fractional order $\alpha, \beta$ and the
initial data $\zeta$. To emphasize the dependence of the solution $u$ on these given data,
let us write it by $u_{\zeta, \alpha, \beta}$.
We have the following theorem.

\begin{theorem}
Let $\alpha\in (0, 1)$, $\beta\in (0, +\infty)$, $s\in [0, \beta/2]$, $\nu<\alpha/2$,
and let $\zeta$ be the initial data
defined in (\ref{initial value}) such that $\zeta \in D(A^{\beta/2})$.
Let the source function $f$ satisfy  \textbf{Assumption F1}
and $g \in C\left([0, T], D(A^s)\right)$ with $g$ defined in (\ref{g-condition}).
~
~
Then, for any $\m> 2 \dnorm{\zeta}_{\beta/2}$, we have

\paragraph{(i).} (Local existence) There exists a $T_\m>0$ such that the FIVP has a unique mild solution
$u_{\zeta, \alpha, \beta}$ which  belongs to $C([0, T_\m]; D(A^s))$.

\paragraph{(ii).} (Uniqueness) If $V,W\in C([0,T];D(A^s))$ are  solutions of \eqref{mild-solution} on $[0,T]$
then $V=W$.

\paragraph{(iii).} (Maximal existence) Let
$$T_{\zeta, \alpha, \beta}=\sup\{T>0:\ \eqref{mild-solution}\ \text{has a unique solution on}\ [0,T]\}.$$
Then the equation \eqref{mild-solution} has a unique solution $u_{\zeta,\alpha,\beta}\in C([0,T_{\zeta,\alpha,\beta});D(A^s))$. Moreover,
 we have either $T_{\zeta, \alpha, \beta} = +\infty$ or $T_{\zeta, \alpha, \beta}<+\infty$
and $\| u_{\zeta, \alpha, \beta}(t)\|_s \to \infty$ as $t \to T_{\zeta, \alpha, \beta}^-$.
Besides, if $u_{\zeta, \alpha, \beta} \in B_{s, T}(\m)$
then
\[
   \| u(t) \|_s^2
\le
   2 \Gamma(1-2\nu) \| g\|_{s, t}^2
   E_{\alpha-2\nu,1-2\nu} \left(2\theta^{2s-\beta} L^2(\m) t^{\alpha-2\nu}\right)
    ,
\]
for any $t\in [0,T]$.
\end{theorem}

\begin{proof}
Before proving the theorem, we set up some notations.
We will use Theorem \ref{global-existence} to prove Part (i). For $M>0$, we put
\[
  f_{M}(t,v)
  =
  f\left(t,\frac{M v}{\max\{M,\| v\|_s\}}\right)\ \text{for}\ v\in D(A^s).
\]
Verifying directly, we can prove that the function $f_M$ is global Lipschitz with respect to the variable $v$, i.e.,
\[
  \dnorm{ f_M(t,w_1)-f_M(t, w_2) }
  \leq
  \kappa_0(M) t^{-\nu} \dnorm{ w_1-w_2}_s
  \ \
  \text{for all}\ w_1,w_2\in D(A^s).
\]
We consider the problem of finding $U\in C([0,T],D(A^s))$ satisfying
\\
\begin{equation}
\label{M-mild-solution}
U(t)
 =
     \fm_\alpha \big(- t^\alpha A^\beta \big) \zeta
    +
    \int_0^t
       \E_{\alpha, \beta}(A, t, \tau)
       f_M(\tau, U(\tau))
    \ud \tau.
\end{equation}
From Theorem \ref{global-existence}, for any $T>0$, the equation \eqref{M-mild-solution} has
 a unique solution $U_{M,T}\in C([0,T],D(A^s))$.

\paragraph{(i).} For any $m>0$, we put $   \m = 2 \| \zeta\|+ m$.
 Since $U_T(0)=\zeta$, we can use the continuity of $U_T$  to find a constant $T_\m\in (0,T]$ such that $\sup_{0\leq t\leq T_\m}\| U_{\m,T}(t)\|_s\leq \m$.
In this case $f_\m(t,U_{\m,T}(t))=f(t,U_{\m,T}(t))$ for all $t\in [0,T_\m]$ and $U_{\m,T}(t)$ satisfies \eqref{mild-solution} for $t\in[0,T_\m]$.

\paragraph{(ii).} If $V,W\in C([0,T];D(A^s))$ are solutions of \eqref{mild-solution}, we denote
 $$\mu=\max\{\sup_{0\leq t\leq T}\| V(t)\|_s, \sup_{0\leq t\leq T}\| W(t)\|_s\}$$
and consider the equation
\begin{equation}
\label{Mu-mild-solution}
U(t)
 =
     \fm_\alpha \big(- t^\alpha A^\beta \big) \zeta
    +
    \int_0^t
       \E_{\alpha, \beta}(A, t, \tau)
       f_\mu(\tau, U(\tau))
    \ud \tau.
\end{equation}
From Theorem \ref{global-existence}, the equation \eqref{Mu-mild-solution} has a unique solution
$U_{\mu,T}\in C([0,T]; D(A^s))$.  Since $\| V(t)\|_s,\| W(t)\|_s\leq \mu$ for $t\in[0,T]$, we have
$f(t,V(t))=f_\mu(t,V(t)),\ f)t,W(t))=f_\mu(t,W(t))$. Hence, $V,W$ satisfies \eqref{Mu-mild-solution}. By
Theorem \ref{global-existence}, we have $V=U_{\mu,T}=W$.

\paragraph{(iii).} For every $T\in (0,T_{\zeta,\alpha,\beta})$, the equation \eqref{mild-solution} has a unique solution $U_T\in C([0,T];D(A^s)$. From Part (ii), for $T_1,T_2\in (0,T_{\zeta,\alpha,\beta})$, $T_1<T_2$, we have $U_{T_1}(t)=U_{T_2}(t)$ for $t\in [0,T_1]$. Hence, we can put $u_{\zeta,\alpha,\beta}(t)=U_T(t)$ for all
$t\in [0,T], T\in (0,T_{\zeta,\alpha,\beta})$. The function $u_{\zeta,\alpha,\beta}$ is the unique solution of \eqref{mild-solution} on $(0,T_{\zeta,\alpha,\beta})$.

We prove the second result of Part (iii). Assume by contradiction that $T_{\zeta,\alpha,\beta}<\infty$ and $\| u_{\zeta,\alpha,\beta}(t)\|_s\leq M$ for every $t\in [0,T_{\zeta,\alpha,\beta})$.
We consider the equation
 \begin{equation}
\label{MM-mild-solution}
U(t)
 =
     \fm_\alpha \big(- t^\alpha A^\beta \big) \zeta
    +
    \int_0^t
       \E_{\alpha, \beta}(A, t, \tau)
       f_M(\tau, U(\tau))
    \ud \tau.
\end{equation}
From Theorem \ref{global-existence}, the equation \eqref{MM-mild-solution} has a unique solution
$U_{M,\delta+T_{\zeta,\alpha,\beta}}$. From Part (ii) we have $u_{\zeta,\alpha,\beta}(t)=U_{M,\delta+T_{\zeta,\alpha,\beta}}(t)$ for every $t\in [0,T_{\zeta,\alpha,\beta})$. Since $U_{M,\delta+T_{\zeta,\alpha,\beta}}
\in C([0,\delta+T_{\zeta,\alpha,\beta}]; D(A^s))$, we can find a constant $\delta'\in (0,\delta)$ such that
$\| U_{M,\delta+T_{\zeta,\alpha,\beta}}(t)\|_s\leq M$ for $t\in [0,\delta'+T_{\zeta,\alpha,\beta}]$.
Hence the equation
\eqref{mild-solution} has a unique solution on $[0,T_{\zeta,\alpha,\beta}+\delta']$. It follows that
$T_{\zeta,\alpha,\beta}+\delta'\leq T_{\zeta,\alpha,\beta}$, which is a contradiction.

Finally, the proof of the last inequality of the theorem is similar to the inequality (\ref{global-bound}).
Hence we omit it.
This completes the proof of the theorem.
\end{proof}

In the next theorem, we state some stability of solution of the initial problem with respect to
the fractional orders and the initial data.
We have the following result.

\begin{theorem}
\label{thrm2}

Let $0<\alpha_*<\alpha^*<2$, $\alpha^*<2\alpha_*$, $0<\beta_*<\beta^*$, and $\Delta$ as in (\ref{delta}).
Let $(\alpha, \beta), (\alpha_k, \beta_k)\in \Delta$
such that $(\alpha_k, \beta_k) \to (\alpha, \beta)$,
and let $\zeta, \, \zeta_k \in D(A^{\beta^*/2})$ such that $\zeta_k \to \zeta$ in $D(A^{\beta^*/2})$
as $k\to\infty$.
~
~
Let the source function $f$ satisfy the {\bf Assumption F1} for every $s\in [\beta_*/2,\beta^*/2]$ such that
\[
   L(\m)
=
   \sup_{s\in [\beta_*/2,\beta^*/2]}L(s,\m)<\infty \ \ \text{for every}\ \m>0.
\]
Suppose that $f(.,0) \in \cc^\alpha(\ct) $ for any $\alpha \in [\alpha_*, \alpha^*]$
and for every $\ct \in (0, +\infty)$,
then, for $T\in (0, T_{\zeta, \alpha, \beta}]$,
there exist a number $\m_T>0$ and a number $k_T$ large enough such that
$T \le T_{\zeta_k, \alpha_k, \beta_k}$
and
$
  u_{\zeta, \alpha, \beta}, \, u_{\zeta_k, \alpha_k, \beta_k}
  \in
  B_{\min\{\beta/2, \beta_k/2\}, T}(\m_T)$
for any $k \ge k_T$.
In addition, the following results hold.

\paragraph*{(i).}

If $p \in [\beta_*/2, \beta/2)$, then
\begin{equation}
\label{lim1}
   \lim_{k\to \infty}
   |u_{\zeta_k, \alpha_k, \beta_k} - u_{\zeta, \alpha, \beta}|_{p, T} = 0.
\end{equation}

\paragraph*{(ii).}

If $\beta_k \to \beta^-$ as $k \to \infty$, then
\begin{equation}
\label{lim2}
   \lim_{k\to \infty}
   |u_{\zeta_k, \alpha_k, \beta_k} - u_{\zeta, \alpha, \beta}|_{\beta_k/2, T} = 0.
\end{equation}

\paragraph*{(iii).}

If we suppose further that $\zeta_k, \zeta \in D\left( A^{\beta^*/2+r_1}\right)$ such that
$\zeta_k \to \zeta$ in $D\left( A^{\beta^*/2+r_1}\right)$ for some $r_1>0$.
We also suppose that $f(t, \vartheta) \in C([0, T], D\left( A^{r_2}\right))$
for some $r_2>0$ and for any $\vartheta \in B_{s,T}(\m_T)$.
Then, there exists a constant $A$ independent of $\zeta, \zeta_k$ such that
\begin{equation}
\label{lim3}
   \snorm{u_{\zeta_k, \alpha_k, \beta_k} - u_{\zeta, \alpha, \beta}}_{s, T}
\le
   A\| \zeta - \zeta_k\|
   +
   B  (|\alpha - \alpha_k| +|\beta - \beta_k|) ^{\frac{\gamma_2}{2(\gamma_1+\gamma_2+2)}},
\end{equation}
where $B=B(\alpha_*,\alpha^*, \beta_*, \beta^*, T)$,
$\gamma_1 = \max\{ \beta^*+2(s-r_1),2(s-r_2), 0\}$, and $\gamma_2=\min\{\beta^*+2(r_1-s), 2r_2\}$.
\end{theorem}

\begin{remark}
Theorem \ref{thrm2} showed that
if $\alpha \to 1^-, \beta \to 1$ then the solution of the fractional equation (\ref{main})--(\ref{initial value})
tend to the solution of classical equation
\[
u_t = A u + f(t, u)
.
\]
\end{remark}

\begin{proof}
To highlight the core of the proof, we will state three complementary results.
Readers can find the proofs of these results
in the Appendices \ref{app-step1}, \ref{app-step2}, and \ref{app-step3},
respectively.
In these results, let us put
\begin{eqnarray}
\nn
\F_{\zeta, \alpha, \beta, A}(v)(t)
 &=&
    \fm_\alpha \big(- t^\alpha A^\beta \big) \zeta
    +
    \int_0^t
       \E_{\alpha, \beta}(A, t, \tau)
       f(\tau, v(\tau))
    \ud \tau,
\end{eqnarray}
where $\E_{a, b}(., t, \tau)$ defined in Lemma \ref{Q-lemma},
and denote
$
u_{\zeta, \alpha,\beta}, u_{\xi, \wt{\alpha}, \wt{\beta}}
$ and
$u_{\zeta, \wt{\alpha}, \wt{\beta}}
$
the solutions of the problems
$
\F_{\zeta, \alpha, \beta, A}(u) = u
,
\F_{\xi, \wt{\alpha}, \wt{\beta}, A}(w) = w
$
and
$
\F_{\xi, \wt{\alpha}, \wt{\beta}, A}(v) = v
$,
respectively.
\begin{lemma} \label{Step1}

Let $\zeta, \xi \in D(A^s)$ be two initial data with $s \in [\beta_*/2, \beta^*/2]$,
and
let
$
   \alpha, \, \wt{\alpha} \in [\alpha_*, \alpha^*]
$,
$
   \beta, \, \wt{\beta}\in [\beta_*, \, \beta^*]
$.
Assume that
$
u_{\xi, \wt{\alpha}, \wt{\beta}}, u_{\zeta, \wt{\alpha}, \wt{\beta}} \in B_{s, T}(\m)
$
for any
$
T \in
\left(0, \min \left\{T_{\zeta, \wt{\alpha}, \wt{\beta}} , T_{\xi, \wt{\alpha}, \wt{\beta}} \right\}\right]
$.
Then, there exists $P_1$ independent of $\zeta - \xi$ such that
\begin{equation*}
\dnorm
{
   u_{\xi, \wt{\alpha}, \wt{\beta}}(t)
  -
   u_{\zeta, \wt{\alpha}, \wt{\beta}}(t)
}_s
\le
   P_1 \dnorm{\zeta - \xi}_s
\end{equation*}
for every $t\in [0, T]$.
\end{lemma}

\begin{lemma} \label{Step2}
~
Let
$
   \alpha, \, \wt{\alpha} \in [\alpha_*, \alpha^*]
$
,
$
   \beta, \, \wt{\beta}\in [\beta_*, \, \beta^*]
$
and $T \in \big(0, \min\{T_{\zeta, \alpha, \beta} , T_{\zeta, \wt{\alpha}, \wt{\beta}} \}\big]$.
Assume that $\zeta \in D(A^{\beta^*/2})$
and
$
u_{\zeta, \wt{\alpha}, \wt{\beta}}
,
u_{\zeta, \alpha, \beta}
\in
B_{s, T}(\m)
$
with $s\in [\beta_*/2, \min\{\beta/2, \wt{\beta}/2\}]$.
Then, for any $\epsilon>0$, there exist two constants $P, \, P_\epsilon>0$
which are independent of  $\alpha-\wt{\alpha}, \, \beta-\wt{\beta}$ and $t$
such that
\begin{equation*}
\dnorm
{
   u_{\zeta, \wt{\alpha}, \wt{\beta}}(t)
  -
   u_{\zeta, \alpha, \beta}(t)
}_s
\le
   P
   \left(
      \epsilon + P_\epsilon \left(|\alpha - \wt{\alpha}| + |\beta - \wt{\beta} |\right)
   \right)^{1/2},
\end{equation*}
for every $t\in [0, T]$.
\end{lemma}

\begin{lemma} \label{Step3}
~
Let
$
\alpha, \, \wt{\alpha} \in [\alpha_*, \alpha^*]
$,
$
\beta, \, \wt{\beta} \in [\beta_*, \, \beta^*]
$
and
$
T \in \left(0, \min\{T_{\zeta, \alpha, \beta} , T_{\zeta, \wt{\alpha}, \wt{\beta}} \}\right]
$.
Assume that
$
u_{\zeta, \wt{\alpha}, \wt{\beta}}
,
u_{\zeta, \alpha, \beta}
\in B_{s, T}(\m)
$
with $s \in [\beta_*/2, \min\{\beta/2, \wt{\beta}/2\}]$
and
$\zeta \in D\left(A^{\beta^*/2+r_1}\right)$ for some $r_1>0$.
We suppose further that
$f(t, w) \in C([0, T], D\left(A^{r_2}\right))$
for any $w \in B_{s, T}(\m)$,
then, there exists a constant $Q_0>0$
which is independent of  $\alpha-\wt{\alpha}, \, \beta-\wt{\beta}, \, N$ and $t$ such that
\begin{equation*}
\dnorm
{
   u_{\zeta, \wt{\alpha}, \wt{\beta}}(t)
  -
   u_{\zeta, \alpha, \beta}(t)
}_s
\le
   Q_0 \left(2^{\gamma_1+2} +1\right)
   \left(
       |\alpha - \wt{\alpha}| + |\beta - \wt{\beta} |
   \right)^{\frac{\gamma_2}{2(\gamma_1+\gamma_2+2)}},
\end{equation*}
for every $t\in [0, T]$.
Herein $\gamma_1 = \max\{ \beta^*+2(s-r_1), 2(s-r_2), 0\}$,
and $\gamma_2=\min\{\beta^*+2(r_1-s), 2 r_2\}$.
\end{lemma}

Using Lemmas \ref{Step1}--\ref{Step3},
we will prove the results of the Theorem.
To this aim, let us fix $T \in (0, T_{\zeta, \alpha, \beta}]$.
We also set $\cl=\snorm{u_{\zeta, \alpha, \beta}}_{\beta/2, T}$,
$
   \tau_k
 =
   \sup
   \{
      \tau \in [0, T_{\zeta_k, \alpha_k, \beta_k}]:
      \snorm{u_{\zeta_k, \alpha_k, \beta_k}}_{\min\{\beta/2, \beta_k/2\}, \tau}
\le
    \left(\max\big\{1, \theta^{\beta_*-\beta^*} \big\}+1\right) \cl
   \}
$
and $T_k=\min\{ \tau_k, \, T\}$ for $k \in \mathbb{N}$.
\\
Using the triangle inequality and Lemmas \ref{Step1}--\ref{Step2}, we obtain
\begin{eqnarray*}
\lefteqn
{
\dnorm
{
   u_{\zeta_k, \alpha_k, \beta_k}(t)
   -
   u_{\zeta, \alpha, \beta}(t)
}_{\min\{\beta_k/2, \, \beta/2\}}
}
\\
&\le &
\dnorm
{
   u_{\zeta_k, \alpha_k, \beta_k}(t)
   -
   u_{\zeta, \alpha_k, \beta_k}(t)
}_{\min\{\beta_k/2, \, \beta/2\}}
+
\dnorm
{
   u_{\zeta, \alpha_k, \beta_k}(t)
   -
   u_{\zeta, \alpha, \beta}(t)
}_{\min\{\beta_k/2, \, \beta/2\}}
\\
& \le &
   P_1\dnorm{\zeta - \zeta_k}_{\beta^*/2}
  +
   P\left(\epsilon + P_\epsilon(|\alpha - \alpha_k| + | \beta - \beta_k|) \right)^{1/2},
\end{eqnarray*}
for any $t \in [0, T_k]$.
In addition, we note that
\begin{equation}
\label{p-q-norm}
   \dnorm{w}_p \le \theta^{p-q} \dnorm{w}_q
\
\text{for any}
\
 0 \le p \le q.
\end{equation}
Consequently,
\begin{eqnarray*}
\lefteqn
{
   \dnorm{u_{\zeta_k, \alpha_k, \beta_k}(t)}_{\min\{\beta_k/2, \, \beta/2\}}
}
\\
& \le &
   \dnorm{u_{\zeta, \alpha, \beta}(t)}_{\min\{\beta_k/2, \, \beta/2\}}
   +
   \left(
      P_1\dnorm{\zeta - \zeta_k}_{\beta^*/2}
      +
      P\left(\epsilon + P_\epsilon(|\alpha - \alpha_k| + | \beta - \beta_k|) \right)^{1/2}
   \right)
 \\
 &< &
   \max \left\{1, \theta^{\beta_*/2-\beta^*/2}\right\}\cl + \cl
=
   \left( \max \left\{1, \theta^{\beta_*/2-\beta^*/2}\right\}+1\right)\cl,
\end{eqnarray*}
for any $t \in [0, T_k]$ and $k$ large enough.
So far, from the definition of $T_k$, we deduce
$T_k =\min\{\tau_k, T\} < \tau_k$
or
$T_j=T$
for $k\ge k_T$ with some $k_T$ large enough.
We can choose $\m_T=\left(\max \left\{1, \theta^{\beta_*/2-\beta^*/2}\right\}+1\right)\cl$,
then
$
  u_{\zeta, \alpha, \beta}, \, u_{\zeta_k, \alpha_k, \beta_k}
  \in
  B_{\min\{\beta/2, \beta_k/2\}, T}(\m_T)
$.

From the latter result, we can verify directly the main results (\ref{lim1}), (\ref{lim2}), (\ref{lim3})
of the theorem.
~
In fact, if $p\in [\beta_*/2, \beta/2)$ then with $k$ large enough, we have $\beta_k/2 \ge p$.
~
Hence, we can combine Lemma \ref{Step1}, Lemma \ref{Step2} with (\ref{p-q-norm}) to obtain (\ref{lim1}).
~
We also use Lemma \ref{Step1} and Lemma \ref{Step2} to deduce (\ref{lim2}).
~
Finally, combining Lemma \ref{Step1} with Lemma \ref{Step3}, we obtain (\ref{lim3}).
This completes the core of the proof.
~
\end{proof}

\section{The FFVP}

This section is  devoted to the study of  existence, and uniqueness of the solution of the FFVP for $t>0$.
In the case the solution is unique, we investigate the stability of solution of the problem
with respect to perturbed fractional orders and the final data.
For $t=0$, we will analyze the ill-posedness of the problem,
after that, we propose a method to regularize this problem.

Firstly, for brevity, from now on, we use the notation $Q_{\alpha,\beta,A}(u))(t)$ to denote the quantity
$Q_{\alpha,\beta,A}(f(.,u))(0,t)$ defined in Lemma \ref{Q-lemma}. For convenience, we
write again the formula of the quantity
\[
Q_{\alpha,\beta,A}(u)(t)
=
\int_{0}^{t}\mathcal{E}_{\alpha,\beta}(A,t,\tau)f(\tau,u(\tau)) \ud\tau.
\]
Using the Fourier series and the Laplace transform,
we can rewrite problem (\ref{main}) and (\ref{final value})
into the following integral equation
\begin{equation}
u(t)
  =
P_{\alpha, \beta}(A, t) G_{\varphi, \alpha, \beta, A}(u)
+
Q_{\alpha, \beta, A} (u)(t),
\label{solution-bw-nonlinear}
\end{equation}
where
\[
   G_{\varphi, \alpha, \beta, A}(u)
 =
   \varphi- Q_{\alpha, \beta, A}(u)(T)
,
\
P_{\alpha, \beta}(A, t)
=
E_{\alpha} \big(-A^\beta t^\alpha \big)
E^{-1}_{\alpha} \big(- A^\beta T^\alpha \big)
.
\]
~
~
~
Before stating  the main results of this part,
we provide some properties of the functions $Q_{\alpha, \beta, A}$ and $G_{\varphi, \alpha, \beta, A}$
in the following lemma.

\begin{lemma}

\label{QG-es}
 Let $\beta> 0$, $s \in [0, \beta/2]$, $\rho\ge \alpha$, and let $\nu<1/2-\rho$ and $\nu \le \alpha/2$.
 Let $f$ satisfy the {\bf Assumption F1} such that
\[
    \kappa=\sup_{\mathcal{M}>0}L(s,\mathcal{M})<\infty.
\]
 Let $w_1, w_2 \in C_{s, \rho}(T)$,
 and $\varphi, \wt{\varphi} \in D(A^s)$.
 We assume that
\[
     \Theta_\alpha(t)
=
     \int_0^t (t-\tau)^{\alpha-1} \dnorm{f(\tau, 0)}^2 \ud \tau
< + \infty
\]
and put
\[
E_0
=
\frac{1}{\Gamma(\alpha)} B(\alpha, 1-2\rho-2\nu)
=
\frac{\Gamma(1-2\rho-2\nu)}{\Gamma(1+\alpha-2\rho-2\nu)}
.
\]
~
Then we have
~
\\
~
(1).
$
\displaystyle
   \| Q_{\alpha, \beta, A}(w_1)(t) - Q_{\alpha, \beta, A}(w_2)(t)\|_s
\le
    \kappa \theta^{s-\theta/2} E_0^{1/2}
    \tnorm{w_1-w_2)}_{s, \rho} t^{\alpha/2-\rho-\nu}
$.
~
\\
~
(2).
$
\displaystyle
    \dnorm{G_{\varphi, \alpha, \beta, A}(w_1) - G_{\wt{\varphi}, \alpha, \beta, A}(w_2)}_s
\le
    \| \varphi - \wt{\varphi} \|_s
    +
     \kappa \theta^{s-\theta/2} E_0^{1/2}
     T^{\alpha/2-\rho-\nu}
    \tnorm{w_1-w_2)}_{s, \rho}
$.
~
\\
~
(3).
$
\displaystyle
   \dnorm{Q_{\alpha, \beta, A}(w_1)(t)}_s^2
 \le
2\theta^{2s-\beta}
\left(
\frac{\Theta_\alpha(t)}{\Gamma(\alpha)}
+
\kappa^2 E_0
\tnorm{w_1}_{s,\rho}^2 t^{\alpha-2\rho-2\nu}
\right)
$.
~
\\
~
(4).
$
\displaystyle
\dnorm{G_{\varphi, \alpha, \beta, A}(w_1)}_s
\le
\dnorm{\varphi}_s
+
\sqrt{2}\theta^{s-\beta/2}
\left[
\left(\frac{\Theta_\alpha(T)}{\Gamma(\alpha)}\right)^{1/2}
+
\kappa
E_0^{1/2}
\tnorm{w_1}_{s,\rho}
T^{\alpha-2\rho-2\nu}
\right]
$.
~
~
\\
~
(5).
If we suppose further that $f(t, w_1(t)) \in C\left(0, T], D(A^r) \right)$
and $\varphi \in D(A^{s+r})$,
then
\[
   G_{\varphi, \alpha, \beta, A}(w_1) \in D(A^{s+r}).
\]
\end{lemma}

 \begin{proof}
 See appendix \ref{app-QG-es}.
 \end{proof}

\subsection{Existence, uniqueness results for $t>0$}

In this part, we use the Krasnoselskii fixed point theorem and the contraction principle
to give the existence, and  uniqueness  of solution of the FFVP.
In fact, we have the following results.

 \begin{theorem}
\label{existence-FVP}

Let $\beta> 0$, $s \in [0, \beta/2]$, $\rho\ge \alpha$, and let $\nu<1/2-\rho$ and $\nu \le \alpha/2$.
Let $f$ satisfy the \textbf{Assumption F1},
and $\varphi \in D\left(A^s\right)$.
We put
\[
    \kappa=\sup_{\mathcal{M}>0}L(s,\mathcal{M}).
\]

(i).
If we suppose further that there exists a  non-negative integrable function $h: [0, T] \to \bR$
and a positive number $M$ such that
\[
    \dnorm{f(t, w(t))} \le M t^{-\varrho} + h(t)
\]
for any $w\in D(A^s)$, for some $\varrho \le \alpha/2+\rho$ and $\varrho<1/2$
and
\[
  P_0=\sup_{0 \le t \le T}t^{2\rho}\int_0^t (t-\tau)^{\alpha-1} h(\tau) \ud \tau <+\infty.
\]
Then there exist $K_0=K_0(\alpha, \beta)>0$
such that for $\kappa<K_0$ the FFVP has at least one solution
in $C_{s, \rho}(T)$.

(ii).
If
\begin{equation}
\label{theta}
   \Theta_T(\alpha)
=
   \sup_{t \in (0, T]}
   t^{2\rho}
   \Theta_\alpha(t)
<+\infty
.
\end{equation}
Then there exists $E=E(\alpha, \beta)>0$
such that for all
$\kappa<K_0/(1+1/E)$ with $K_0$ defined in part (i)
the FFVP has a unique solution in $C_{s, \rho}(T)$,
say $u$.
In addition, if $\kappa<K_0/\left(\sqrt{2}(1+1/E)\right)$ then we have the upper bound estimate
\begin{equation}
\label{upper-es}
   \tnorm{u}_{s, \rho}
\le
   (1-L)^{-1}
   \left(
      E T^\rho \dnorm{\varphi}_s
      +
      \sqrt{2}(1+E) \theta^{s-\beta/2}
     \left(\frac{\Theta_T(\alpha)}{\Gamma(\alpha)} \right)^{1/2}
   \right)
,
\end{equation}
where
$ L = 1-\sqrt{2} \kappa (1+1/E)/K_0$.
\end{theorem}


\begin{proof}
(i).
Firstly, according to Lemma \ref{trong1},
there exists a constant $E=E(\alpha, \beta)$
such that
\begin{equation}
\label{P-es}
   P_{\alpha, \beta}(\lambda, t)
   \le
   E (T/t)^\alpha
\end{equation}
for any $\lambda \ge \theta$.
Let us define
\[
\Lambda_0
=
   \frac{2\theta^{2s-\beta}}{\Gamma(\alpha)}
   \left(
      M^2 B(\alpha, 1-2\varrho)  T^{\alpha+2\rho-2\varrho}
     +
      P_0
   \right),
\]
and
\[
   \Omega_1
=
   \left\{
      w \in C_{s, \rho}(T): \tnorm{w}_{s, \rho}
\le
      E T^\rho \dnorm{\varphi}_s
      +
      (1+E) \Lambda_0^{1/2}
   \right\}
.
\]
The proof is divided in three steps.
\begin{itemize}

\item[•]
\textbf{Step 1.}
The operator $Q_{\alpha, \beta, A}$ is completely continuous.

\item[•]
\textbf{Step 2.}
The operator $B(u)(t)=P_{\alpha, \beta}(A, t) G_{\varphi, \alpha, \beta, A}(u)$
is a contraction  in $C_{s, \rho}(T)$.

\item[•]
\textbf{Step 3.}
$
B(u) +Q_{\alpha, \beta, A}(v) \in \Omega_1
$
for all $u, v \in \Omega_1$.
\end{itemize}

\ni{\bf Proof of Step 1.}
We will us the same method of the proof of Theorem \ref{existence-thm}.
Indeed, using Lemma \ref{QG-es}, part (1), we have $Q_{\alpha, \beta, A}$ continuous in $C_{s, \rho}(T)$.
Applying Lemma \ref{Q-lemma} with notice that $\sup_{\lambda \ge \theta}\lambda^{2s-\beta} H_0(\lambda, 0, t) \le \theta^{2s-\beta}$,
we have
\begin{eqnarray*}
   \dnorm{Q_{\alpha, \beta, A}(u)(t)}_s^2
&\le &
   \frac{\theta^{2s-\beta}}{\Gamma(\alpha)}
   \int_0^t (t-\tau)^{\alpha-1} \dnorm{f(\tau, u)}^2 \ud \tau
\nn
\\
&\le &
   \frac{2\theta^{2s-\beta}}{\Gamma(\alpha)}
   \left(
      M^2 \int_0^t (t-\tau)^{\alpha-1} \tau^{-2\varrho}\ud \tau
      +
      \int_0^t (t-\tau)^{\alpha-1} h(\tau) \ud \tau
   \right)
\nn
\\
&=&
   \frac{2\theta^{2s-\beta}}{\Gamma(\alpha)}
   \left(
      M^2 B(\alpha, 1-2\varrho)  t^{\alpha-2\varrho}
     +
      \int_0^t (t-\tau)^{\alpha-1} h(\tau) \ud \tau
   \right)
\end{eqnarray*}
for any $u \in D(A^s)$.
This gives
\begin{eqnarray}
\label{bound-Q}
   \tnorm{Q_{\alpha, \beta, A}(u)}_{s, \rho}^2
& \le &
   \frac{2\theta^{2s-\beta}}{\Gamma(\alpha)}
   \left(
      M^2 B(\alpha, 1-2\varrho)  T^{\alpha+2\rho-2\varrho}
     +
      P_0
   \right)
\nn
\\
&:=&
   \Lambda_0
\end{eqnarray}
So $Q_{\alpha, \beta, A}(C_{s, \rho}(T))$ is bounded.
For $t_1<t_2$, we have
\[
   \dnorm{Q_{\alpha, \beta, A}(w)(t_1)-Q_{\alpha, \beta, A}(w)(t_2)}_s
\le
   \dnorm{
   R_{\alpha,\beta,A}(f(.,w))(t_1,t_2)}_s
+
   \dnorm{ Q_{\alpha,\beta,A}(f(.,w))(t_1,t_2)}_s
,
\]
where $R_{\alpha,\beta,A}(f(., w)), Q_{\alpha,\beta,A}(f(., w))(t_1, t_2)$ defined in Lemma \ref{Q-lemma}.
Hence, by similar method in Theorem \ref{existence-thm},
we can directly verify that $Q_{\alpha, \beta, A}(\Omega_1)$ is equicontinuous.

\ni{\bf Proof of Step 2.}
Using (\ref{P-es}), Lemma \ref{QG-es} (part (2)) and by direct computation, we have
\begin{eqnarray*}
   \dnorm{B(w_1)(t)-B(w_2)(t)}_s
&\le &
   E (T/t)^\alpha
   \dnorm
   {
      G_{\varphi, \alpha, \beta, A}(w_1)
      -
      G_{\varphi, \alpha, \beta, A}(w_2)
   }_s
\\
&\le &
   \kappa \theta^{s-\beta/2} E_0^{1/2} E T^{3\alpha/2-\rho-\nu}
   \tnorm{w_1-w_2)}_{s, \rho} t^{-\alpha}
   ,
\end{eqnarray*}
where $E_0$ defined as in Lemma \ref{QG-es}.
This gives
\begin{eqnarray}
\label{B-es}
   \tnorm{B(w_1)-B(w_2}_{s, \rho}
&\le &
   \kappa \theta^{s-\beta/2} E_0^{1/2} E T^{\alpha/2-\nu}
   \tnorm{w_1-w_2)}_{s, \rho}
   .
\end{eqnarray}
If we put
\begin{equation}
\label{df-K0}
K_0=\left( \theta^{s-\beta/2} E_0^{1/2} E T^{\alpha/2-\nu}\right)^{-1}
\end{equation}
then $B$ is a contraction  in $C_{s, \rho}(T)$ for any $0 \le \kappa<K_0$.

\ni{\bf Proof of Step 3.}
By (\ref{bound-Q}), we have
\begin{eqnarray*}
   \tnorm{B(u)}_{s, \rho}
& \le &
   \sup_{0<t \le T}
   t^{\rho}
   \dnorm{P_{\alpha, \beta}(A, t) \varphi}_s
+
   \sup_{0< t \le T}
   t^{\rho}
   \dnorm{P_{\alpha, \beta}(A, t) Q_{\alpha, \beta, A}(u)(T)}_s
\\
&\le &
   E T^\rho \dnorm{\varphi}_s +E T^\rho \dnorm{Q_{\alpha, \beta, A}(u)(T)}_s
\\
&\le &
   E T^\rho \dnorm{\varphi}_s +E \dnorm{Q_{\alpha, \beta, A}(u)}_{s, \rho}
\\
&\le &
   E T^\rho \dnorm{\varphi}_s +E \Lambda_0^{1/2}
   .
\end{eqnarray*}
By (\ref{bound-Q}), we can verify that
\[
    \tnorm{B(u)+Q_{\alpha, \beta, A}(v)}_{s, \rho}
 \le
    \tnorm{B(u)}_{s, \rho}+\tnorm{Q_{\alpha, \beta, A}(v)}_{s, \rho}
 \le
    E T^\rho \dnorm{\varphi}_s
    +
    (1+E) \Lambda_0^{1/2}
    .
\]
The last  inequality  shows that
$
B(u)+Q_{\alpha, \beta, A}(v) \in \Omega_1
$
for all $u, v \in \Omega_1$.
This completes the proof of Step 3.
Now, we can use the Krasnoselskii fixed point theorem (see \cite[p.31]{Smart})
to obtain the desired result in  part (i).
\\
\\
(ii).
We define
\[
   \cq_{\varphi, \alpha, \beta, A}(u)(t)
=
   P_{\alpha, \beta}(A, t) G_{\varphi, \alpha, \beta, A}(u)
  +
   Q_{\alpha, \beta, A} (u)(t)
=
   B(u)(t)
  +
   Q_{\alpha, \beta, A} (u)(t)
   ,
\]
where $B(u)(t)=P_{\alpha, \beta}(A, t) G_{\varphi, \alpha, \beta, A}(u)$.
We can use Lemma \ref{QG-es} to verify that $\cq(u) \in C_{s, \rho}(T)$
for any $u\in C_{s, \rho}(T)$.
On the other hand, applying Lemma \ref{QG-es} and (\ref{B-es}), we have
\begin{eqnarray}
\label{Q-w12}
   \tnorm{\cq_{\varphi, \alpha, \beta, A}(w_1)-\cq_{\varphi, \alpha, \beta, A}(w_2)}_{s, \rho}
&\le &
   \tnorm{B(w_1)-B(w_2)}_{s, \rho}
  +
   \tnorm{Q_{\alpha, \beta, A} (w_1)-Q_{\alpha, \beta, A} (w_2)}_{s, \rho}
\nn
\\
& \le &
   \kappa \theta^{s-\beta/2} (E+1) E_0^{1/2}T^{\alpha/2-\nu}
   \tnorm{w_1-w_2)}_{s, \rho}
\nn
\\
&=&
   \kappa\left( 1+ 1/E \right) /K_0
   \tnorm{w_1-w_2)}_{s, \rho}
,
\end{eqnarray}
where $K_0$ is  defined in (\ref{df-K0}).
If
\[
   \kappa
<
   K_0/(1+1/E)
\]
then $\cq_{\varphi, \alpha, \beta, A}$ is a contraction  in $C_{s, \rho}(T)$.
Hence the problem has a unique solution.
Finally, we prove the upper bounded estimate.
Using inequality $a^2+b^2 \le (a+b)^2, (a, b\ge 0)$, by Lemma \ref{QG-es} (part (3)),
one has
\begin{eqnarray}
\label{QF-es}
   \tnorm{\cq_{\varphi, \alpha, \beta, A}(u)}_{s, \rho}
&\le &
   \tnorm{B(u)}_{s, \rho}+\tnorm{Q_{\alpha, \beta, A}(u)}_{s, \rho}
\nn
\\
& \le &
   E T^\rho \dnorm{G_{\varphi, \alpha, \beta, A}(u)}_s
+
   \tnorm{Q_{\alpha, \beta, A}(u)}_{s, \rho}
\nn
\\
& \le &
   E T^\rho \dnorm{\varphi}_s
+
   (1+E) \tnorm{Q_{\alpha, \beta, A}(u)}_{s, \rho}
\nn
\\
& \le &
   E T^\rho \dnorm{\varphi}_s
+
\sqrt{2}(1+E) \theta^{s-\beta/2}
   \left[
       \left(\frac{\Theta_T(\alpha)}{\Gamma(\alpha)} \right)^{1/2}
       +
       \kappa E_0^{1/2} T^{\alpha/2-\nu}
       \tnorm{u}_{s, \rho}
   \right]
\nn
\\
& \le &
   E T^\rho \dnorm{\varphi}_s
+
   \sqrt{2}(1+E) \theta^{s-\beta/2}
   \left(\frac{\Theta_T(\alpha)}{\Gamma(\alpha)} \right)^{1/2}
   +
   L \tnorm{u}_{s, \rho}
   ,
\end{eqnarray}
where $K_0$, $E$ defined in (\ref{df-K0}), (\ref{P-es}) respectively,
$\Theta_T(\alpha)$ defined in (\ref{theta}),
and $L=\sqrt{2} \kappa (1+1/E)/ K_0$.
~
~
If $L<1$
or $\kappa<K_0/ \left(\sqrt{2}(1+1/E)\right)$
then the FFVP has a unique solution $u$ in $C_{s, \rho}(T)$,
and $\cq(u)=u$.
Therefore, (\ref{QF-es}) gives
\[
   \tnorm{u}_{s, \rho}
\le
   (1-L)^{-1}
   \left(
      E T^\rho \dnorm{\varphi}_s
      +
      \sqrt{2}(1+E) \theta^{s-\beta/2}
     \left(\frac{\Theta_T(\alpha)}{\Gamma(\alpha)} \right)^{1/2}
   \right)
.
\]
This completes the proof of part (ii) and the proof of Theorem.
\end{proof}

\subsection{Stability results}
In this part, we investigate the stability of solution of the FFVP with respect to
the fractional orders and the final data.
To emphasize the dependence of solution of the FFVP on the given data $\alpha, \beta$ and $\varphi$,
we denote it by $u_{\varphi, \alpha, \beta}$.
Using the notation, we have the following results.

\begin{theorem}
\label{Stability-FVP}

Let $\alpha_*, \alpha^*$, $\beta_*, \beta^*$ and $\Delta$ as in (\ref{delta}).
Let $(\alpha, \beta) \in \Delta$, $\rho \ge \alpha^*$,
and $\nu <1/2-\rho$.
Let $\varphi \in D\left(A^{\beta^*/2}\right)$,
the source function $f$ satisfies the \textbf{Assumption F1}
for any $s \in [\beta_*/2, \beta^*/2]$.
We suppose that
\[
   \sup_{\alpha_* \le \alpha \le \alpha^*} \Theta_T(\alpha)<\infty
,
\]
where $\Theta_T(\alpha)$ defined in (\ref{theta}),
and
\begin{equation}
\label{Km}
    \kappa=\sup_{\mathcal{M}>0}L(s,\mathcal{M})
<
   K_m
:=
   \min_{\alpha_*\le \alpha \le \alpha^*, \beta_* \le \beta \le \beta^*}
   K_0(\alpha, \beta)/\left(\sqrt{2}(1+1/E(\alpha, \beta))\right)
   \ \  \forall s \in [\beta_*/2, \beta^*/2],
\end{equation}
where $K_0$, $E$ defined in Theorem \ref{existence-FVP}.
Then we have

\paragraph*{(i).}

For
$
   \wt{\alpha} \in [\alpha_*, \alpha^*]
$,
$
   \wt{\beta} \in [\beta_*, \beta^*]
$,
and
$
   \varphi, \, \wt{\varphi} \in D\left(A^{\beta^*/2}\right)
$,
then, we have
\[
\tnorm
{
   u_{\wt{\varphi}, \wt{\alpha}, \wt{\beta}}
   -
   u_{\varphi, \alpha, \beta}
}_{\min\{\beta/2, \wt{\beta}/2\}, \, \rho}
\to 0
\
\
\text{as}
\
\
   (\wt{\varphi}, \wt{\alpha}, \wt{\beta}) \to (\varphi, \alpha, \beta)
.
\]

\paragraph*{(ii).}

For $\wt{\alpha} \in [\alpha_*, \alpha^*], \,  \wt{\beta} \in [\beta_*, \beta^*]$
and
$\varphi, \wt{\varphi} \in D\left(A^{\beta^*/2+r}\right)$ for some $r>0$.
If we suppose further that
$f\left( t, w\right) \in C\left([0, T], D(A^r)\right)$
for any $w\in C_{\beta/2, \rho}(T)$,
then
\[
\tnorm
{
  u_{\wt{\varphi}, \wt{\alpha}, \wt{\beta}}
  -
  u_{\varphi, \alpha, \beta}
}_{\min\{\beta/2, \wt{\beta}/2\}, \, \rho}
\le
   C \| \varphi - \wt{\varphi} \|_{\beta^*/2+r}
+
   D
   \left(
      | \alpha - \wt{\alpha} | + | \beta -\wt{\beta}|
   \right)^{ \frac{r}{2(r+2\beta^*+1)}}
,
\]
where $C, D$ independent of $\alpha, \beta, \wt{\alpha}, \wt{\beta}, N$ and $\varphi, \wt{\varphi}$.
\end{theorem}

\begin{proof}
In the proof of this theorem, we use the following notation
\[
   \cq_{\varphi, \alpha, \beta, A}(u)(t)
=
   P_{\alpha, \beta}(A, t) G_{\varphi, \alpha, \beta, A}(u)
  +
   Q_{\alpha, \beta, A} (u)(t)
.
\]
Base on the assumption (\ref{Km}) and according to part (ii) of Theorem \ref{existence-FVP},
we known that the FFVP has a unique solution $u_{\varphi, \alpha, \beta}$ in $C_{s, \rho}$
which is
the (unique) fixed point of the mapping $\cq_{\varphi, \alpha, \beta, A}(u)$
and has an upper bound estimate (see  (\ref{QF-es}))
\begin{eqnarray}
\label{F-upper}
   \tnorm{u}_{s, \rho}
&\le &
   (1-L)^{-1}
   \left(
      E T^\rho \dnorm{\varphi}_s
      +
      \sqrt{2}(1+E) \theta^{s-\beta/2}
     \left(\frac{\Theta_T(\alpha)}{\Gamma(\alpha)} \right)^{1/2}
   \right)
\nn
\\
&\le &
   (1-\kappa/K_m)^{-1}
   \left(
      E T^\rho \dnorm{\varphi}_{\beta^*/2+r}
      +
      \sqrt{2}(1+E) \theta^{\beta_*-\beta^*/2}
     \left(\frac{\Theta_T(\alpha)}{\Gamma(\alpha)} \right)^{1/2}
   \right)
\end{eqnarray}
since  $L=\sqrt{2} \kappa (1+1/E)/ K_0 \le \kappa/K_m$.
We emphasize that the upper bound above  is independent of  $s$.

For the convenience in writing, we denote
$
  u_{\varphi, \alpha, \beta}
  ,
  u_{\varphi, \wt{\alpha}, \wt{\beta}}
$
and
$
  u_{\wt{\varphi}, \wt{\alpha}, \wt{\beta}}
$
the solution of the nonlinear FFVP the problems
$
    \cq_{\varphi, \alpha, \beta, A}(u)=u
    ,
    \cq_{\varphi, \wt{\alpha}, \wt{\beta}, A}(v)=v
$
and
$
    \cq_{\wt{\varphi}, \wt{\alpha}, \wt{\beta}, A}(w)=w
$,
respectively.
In order to obtain the result of this part, we need the following three  essential Lemmas.
Readers can find the proofs of these Lemmas in the appendices \ref{app-Step1-2}, \ref{app-Step2-2},
and \ref{app-Step3-2}.

\begin{lemma}
\label{Step1-2}
Let
$\left(\wt{\alpha}, \wt{\beta}\right) \in \Delta$,
and $\varphi, \wt{\varphi} \in H^{\beta^*/2+r}$ for some $r \ge 0$.
Then,
\[
\tnorm
{
   u_{\wt{\varphi}, \wt{\alpha}, \wt{\beta}}
   -
   u_{\varphi, \wt{\alpha}, \wt{\beta}}
}_{\gamma, \, \rho}
 \le
\left(1-\kappa/K_m\right)^{-1}
E T^{\rho} \max\left\{1, \, \theta^{\gamma-r-\beta^*/2}\right\}
\| \wt{\varphi} - \varphi\|_{\beta^*/2+r}
,
\]
where $\gamma=\min\{\beta/2, \wt{\beta}/2 \}$.
\end{lemma}

\begin{lemma}
\label{Step2-2}
Suppose that
$\alpha \in [\alpha_*, \alpha^*]$,
and
$\beta, \wt{\beta} \in [\beta_*, \beta^*]$.
Then, there exists a constant
$D$ independent of $\alpha, \, \beta, \, \wt{\alpha}, \, \wt{\beta}, \, N$ such that
\[
\tnorm
{
  u_{\varphi, \wt{\alpha}, \wt{\beta}}
  -
  u_{\varphi, \alpha, \beta}
}_{\gamma, \, \rho}
\le
(1-\kappa/K_m)^{-1}
\left(
      \epsilon
      +
      D N^{2 \beta^*} \ln N
      \left(| \alpha - \wt{\alpha} | + | \beta -\wt{\beta}|\right)^{1/2}
\right),
\]
where $\kappa$ and $K_m$ as in (\ref{Km}),
$\gamma=\min\{\beta/2, \wt{\beta}/2 \}$.
\end{lemma}

\begin{lemma}
\label{Step3-2}
Suppose that $f(t, w) \in C([0, T], H^r)$ for any $w \in C_{\beta/2, \rho}(T)$ and for some $r>0$,
then, for $N$ large enough, we have
\[
\tnorm
{
u_{\varphi, \wt{\alpha}, \wt{\beta}}
-
u_{\varphi, \alpha, \beta}
}_{\gamma, \, \rho}
\le
E_1 N^{ - r }
+
E_2 N^{2 \beta^*} \ln N
\left( | \alpha - \wt{\alpha} | + | \beta -\wt{\beta}| \right)^{1/2},
\]
where $\gamma=\min\{\beta/2, \wt{\beta}/2 \}$,
and $E_1$, $E_2$ independent of
$
\alpha, \, \beta, \, \wt{\alpha}, \, \wt{\beta}, \, N
$.
\end{lemma}

Now, we use the above Lemmas to prove the results of the Theorem.
Combining Lemmas \ref{Step1-2} with \ref{Step2-2}, we obtain the result of part (i).

For part (ii), we can use triangle inequality,
Lemma \ref{Step1-2}, and Lemma \ref{Step3-2} to get that
\begin{eqnarray}
\label{es-final}
\lefteqn{
\tnorm
{
  u_{\varphi, \alpha, \beta}
  -
  u_{\wt{\varphi}, \wt{\alpha}, \wt{\beta}}
}_{\min\{\beta/2, \wt{\beta}/2\}, \, \rho}
}
\nn
\\
&\le &
C \dnorm{\varphi - \wt{\varphi}}_{\beta^*/2+r}
+
E_1 N^{ - r }
+
E_2 N^{2 \beta^*} \ln N
\left(| \alpha - \wt{\alpha} | + | \beta -\wt{\beta}|\right)^{1/2}.
\end{eqnarray}
For $| \alpha - \wt{\alpha} |$, $| \beta -\wt{\beta}|$ small enough,
let us choose
$
N
 =
\left[
\left(
   | \alpha - \wt{\alpha} | + | \beta -\wt{\beta}|
\right)^{ - \frac{1}{2(r+2\beta^*+1)}}
\right] +1
$.
~
~
We note that $\ln N \le N$
and
$
\left(
   | \alpha - \wt{\alpha} | + | \beta -\wt{\beta}|
\right)^{ - \frac{1}{2(r+2\beta^*+1)}}
<
N
\le
2
\left(
   | \alpha - \wt{\alpha} | + | \beta -\wt{\beta}|
\right)^{ - \frac{1}{2(r+2\beta^*+1)}}
$.
Hence, by (\ref{es-final}), we infer
\begin{equation*}
\tnorm
{
u_{\varphi, \alpha, \beta}
-
u_{\wt{\varphi}, \wt{\alpha}, \wt{\beta}}
}_{\min\{\beta/2, \wt{\beta}/2\}, \, \rho}
\le
C \dnorm{\varphi - \wt{\varphi}}_{\beta^*/2+r}
+
D
\left(
   | \alpha - \wt{\alpha} | + | \beta -\wt{\beta}|
\right)^{ \frac{r}{2(r+2\beta^*+1)}},
\end{equation*}
where $D=E_1+ 2^{2\beta^*+1}E_2$.
This completes the proof of Part (ii) and the proof of the Theorem.
\end{proof}

\subsection{Regularization result for $t=0$}

In this part, we study the ill-posedness of the FFVP at $t=0$,
after that, we introduce a regularization method for this problem.
~
~
Now, let us analyze the ill-posedness of the FFVP at $t=0$.
Indeed, when the fractional orders are fixed, solution of
the FFVP is instable. Readers can see in \cite{YangRenLi}.
Hence, we only analyze the instability of solution of the FFVP with respect
to fractional parameters.
To simply, we consider the FFVP in the case of the homogeneous problem
and that
the operator $A$ has  the system  $(\lambda_k,\phi_k)$ of
eigenvalues $\{\lambda_k\}$ and eigenfunctions $\{\phi_k\}$, respectively, with
\[
   0< \lambda_1<\lambda_2<..<\lambda_k<..,
   \, \,
      \lim_{k \to \infty} \lambda_k = \infty.
\]
and the functions $\{\phi_k\}$ being an orthonormal basis of the space $H$.

For $\lambda_n>1$, we put
\[
  \Phi_n
  =
  \frac{\phi_n}{\lambda^\beta_n \ln \lambda_n}.
\]
It is easy to see that $\|\Phi_n\| \to 0$ as $n\to \infty$.
We denote the solution of the homogeneous FFVP corresponding
to the final data $u(T)=\Phi_n$ and the fractional orders $\alpha, \beta$ by
$u_{\Phi_n, \alpha, \beta}$.
If the fractional orders $\alpha, \beta$ are fixed,
using Lemma \ref{trong1}, we have
\[
\dnorm{u_{\Phi_n, \alpha, \beta}}^2
=
\left|
      \frac{1}{\lambda^\beta_n \ln \lambda_n}
      \frac{1}{E_\alpha(-\lambda_n^\beta T^\alpha)}
\right|^2
\le
 \frac{C}{\ln^2 \lambda_n}
\to 0
   \, \,
(n\to \infty).
\]
Now, we perturb the fractional orders by $\beta_n = \beta+ \epsilon_n$
with $\epsilon_n = \frac{2 \ln \ln \lambda_n}{\ln \lambda_n}$, then
we have $\beta_n-\beta \to 0$ as $n\to \infty$.
\[
\dnorm{u_{\Phi_n, \alpha, \beta_n}}^2
=
\frac{1}{\lambda^{2\beta}_n \ln^2 \lambda_n}
\left|
      \frac{1}{E_\alpha(-\lambda_n^{\beta+\epsilon_n} T^\alpha)}
\right|^2.
\]
Since  $\lambda_n^{\epsilon_n} = \ln^2 \lambda_n$,
using Lemma \ref{trong1}, there exist two constants $C_1, C_2$
dependent only on $\alpha$ such that
\[
   \| u_{\Phi_n, \alpha, \beta_n}\|
  \ge
   \frac{1}{\lambda^{\beta}_n \ln \lambda_n}
   \left| \frac{1}{E_\alpha(-\lambda_n^{\beta+\epsilon_n} T^\alpha)} \right|
 \ge
   \frac{ C_1 \lambda_n^{\epsilon_n}}{\ln \lambda_n}
   \to  \infty
\]
as $n \to \infty$.
Hence, the solution of the FFVP is instable.
This example also shows that the difference  from the case of the fractional orders are fixed.
So a method to regularize solution of the FFVP in case of the fractional orders are perturbed is in order.

From our discussion above, we only give a method to
regularize the FFVP at $t=0$.

Let $T^*$ be the number defined in Theorem \ref{Stability-FVP} and $T\in (0, T^*)$.
The final value $u(T)=\varphi$ is given.
Let $\epsilon \in(0,1)$ and $0<\alpha_*<\alpha^*<1$, $0<\beta_*<\beta^*$.
Assume that the measurement data
$
\alpha_\epsilon \in [\alpha_*, \alpha^*],
\,
\beta_\epsilon \in [\max\{\sigma, \beta_*\}, \beta^*]
$
and $\varphi_\epsilon \in H^{\beta^*/2+r}$ for some $r>0$ which satisfy the following conditions
\begin{equation} \label{data-nonlinear}
   | \alpha - \alpha_\epsilon| \le \epsilon, \, |\beta - \beta_\epsilon| \le \epsilon, \,
   \|\varphi - \varphi_\epsilon\|_{\beta^*/2+r} \le \epsilon.
\end{equation}
Before stating the regularized result, we define
\begin{equation}
\label{H-gamma}
C^\rho_\gamma(T)
=
\left\{
w \in \left( C[0, T];H \right):
\dnorm{u(t)-u(0)}_\gamma
\le
E t^\rho
\, \,
\text{for all}
\,\,
t\in[0, T]
\right\}
,
\end{equation}
where $E$ independent of $t$ and $\gamma$.
~
Let us give an example on a class of functions as defined in  (\ref{H-gamma}).

{\bf Example:}
Put $\Omega = (0, 1)$, $H=L^2(\Omega)$, $n \in \mathbb{N}$,
we consider a class of functions $u \in C\left([0, T];  L^2(\Omega)\right)$ such that
\[
\frac{\partial^{k+1}}{\partial x^k \partial t}u(x, t)
\in
L^2\left(0, T; L^2(\Omega) \right)
\, \,
\text{for}
\,
\,
k=\overline{0,n}
.
\]
Then, we have
\[
\frac{\partial^k }{\partial^k x}
\left( u(x, t) - u(x, 0) \right)
 =
\int_0^t
   \frac{\partial^{k+1}}{\partial x^k\partial s}
   u(x, s)
\ud s
\]
for any $k=\overline{0, n}$.
Consequently,
\begin{eqnarray*}
\dnorm{u(., t) - u(., 0) }_n^2
 &=&
\sum_{k=0}^{n}
\int_0^1
\left|
\int_0^t
\frac{\partial^{k+1}}{\partial x^k \partial s}u(x,s)
\ud s
\right|^2\ud x
\nn
\\
& \le &
t
\sum_{k=0}^{m_0}
\int_0^1 \int_0^t
 \left| \frac{\partial^{k+1} }{\partial x^k \partial s}u(x,s) \right|^2
\ud x \ud s
\le
D_0^2 t.
\end{eqnarray*}
~
This implies
\begin{equation*}
\dnorm{u_{\varphi, \alpha, \beta}(., t) - u_{\varphi, \alpha, \beta}(., 0)}_n
\le
D_0 t^{1/2},
 \end{equation*}
 where
 $
 D_0
 =
 \left(
 \sum_{k=0}^n
 \dnorm{\frac{\partial^{k+1} u_{\varphi, \alpha, \beta}}{\partial x^k \partial t}}^2_{L^2(0, T; L^2(\Omega))}
 \right)^{1/2}
$.
Thus,
$
u
\in
C^{1/2}_n(T)
$.
This shows the reasonableness of the definition in (\ref{H-gamma}).

Base on the above notations, we have the following theorem.

\begin{theorem}
Let $f, \, \alpha, \, \beta, \, \gamma$ as in Lemma \ref{QG-es}.
Let $\varphi$ be the final data which is belong to $H^{\beta^*/2+r}$ for some $r>0$,
and let $\alpha_\epsilon, \beta_\epsilon, \varphi_\epsilon$
be measurement which satisfy (\ref{data-nonlinear}).
We assume that $u_{\varphi, \alpha, \beta} \in C^\rho_\sigma(T)$ for some $\rho>0$,
and the assumptions of part (c) of the Theorem \ref{Stability-FVP} holds.
Then, we approximate the initial data
 $u_{\varphi, \alpha, \beta}(0)$
 by
 $u_{\varphi_\epsilon, \alpha_\epsilon, \beta_\epsilon}( t_\epsilon)$
 with $t_\epsilon = \epsilon^{\frac{ r } {2(\alpha^*+\rho)(r+2\beta^*+1)}}$.
More specifically,
we have the following estimate
\[
\|
u_{\varphi_\epsilon, \alpha_\epsilon, \beta_\epsilon}(t_\epsilon)
 -
u_{\varphi, \alpha, \beta}(0)
\|_\sigma
\le
P \epsilon^{\frac{ r \rho } {2(\alpha^*+\rho)(r+2\beta^*+1)}},
 \]
 where $P$ independent of $\epsilon$.
\end{theorem}

\begin{proof}
We put $t_\epsilon = \epsilon^{\frac{ r } {2 (\alpha^*+\rho)(r+2\beta^*+1)}}$.
Using Part (c) of Theorem \ref{Stability-FVP} and (\ref{data-nonlinear}), we get that
\begin{equation}
\label{esitamte-Hm-t}
\dnorm
{
u_{\varphi_\epsilon, \alpha_\epsilon, \beta_\epsilon}(t_\epsilon)
-
u_{\varphi, \alpha, \beta}(t_\epsilon)
}_\sigma
\le
D_1 t_\epsilon^{ - \alpha^*} \epsilon^{\frac{ r } {2(r+2\beta^*+1)}}
\le
D_1\epsilon^{\frac{ r \rho } {2(\alpha^*+\rho)(r+2\beta^*+1)}},
\end{equation}
where $D_1=C+D$ with $C, D$ defined in Part (c) of the Theorem \ref{Stability-FVP}.
On the other hand, since $u_{\varphi, \alpha, \beta} \in C^\rho_\sigma(T)$, this gives
\[
\dnorm
{
u_{\varphi, \alpha, \beta}(t_\epsilon)
-
u_{\varphi, \alpha, \beta}(0)
}_\sigma
\le
E \epsilon^{\frac{ r \rho } {2(\alpha^*+\rho)(r+2\beta^*+1)}}
\]
Combining the latter inequality with (\ref{esitamte-Hm-t}),
we obtain
\[
\dnorm
{
u_{\varphi_\epsilon, \alpha_\epsilon, \beta_\epsilon}( t_\epsilon) - u(0)
}_\sigma
\le
P \epsilon^{\frac{ r \rho} {2(\alpha^*+\rho)(r+2\beta^*+1)}},
 \]
 where $P=D_1+E$.
 This completed the proof of the theorem.
\end{proof}

%
%
%


\section{Proofs}

\subsection{The proof of Lemma \ref{gronwall}.}
\label{gronwall-app}

Put
\[
  Su(t)=v(t) + g(t) \int_0^t (t - \tau )^{\alpha -1} \tau^{-q} u(\tau) \ud \tau.
\]
Using the similar technique as in Theorem \ref{global-existence},
we can prove that there exists $k_0 \in \mathbb{N}$ such that $S^{k_0}$ is contract in $C[0,T]$.
Consequently, there exists a unique $u \in C[0, T]$ such that $u=Su$.

We put $u_0=0, u_{n+1}=Su_{n}$.
The function can be represented by the series $u=\sum_{n=1}^\infty (u_{n+1}-u_{n})$ .
The Weierstrass theorem shows that the series converges in $C[0,T]$ and
\begin{eqnarray*}
   | u(t)|
&\le &
   \| u_1-u_0\|_{C[0,t]}
   \sum_{k=0}^\infty
      \frac{
        \Gamma(1-q)(\| g\|_{C[0,t]}\Gamma(\alpha))^k t^{k(\alpha-q)}
      }{
        \Gamma(k(\alpha-q)-q+1)
      }
\\
&=&
   \Gamma(1-q) \| v\|_{C[0,t]}
   E_{\alpha-q, 1-q} \left(\| g\|_{C[0,t]}\Gamma(\alpha) t^{\alpha-q}\right).
\end{eqnarray*}
Now, we prove the final inequality. Put $w_0=S(w)$, $w_{n+1}=S(w_n)$. Since $g(t)\geq 0$ for $t\in [0,T]$, we have $S(w_1)(t)\leq S(w_2)(t)$ for $w_1(t)\leq w_2(t)$, $t\in [0,T]$. We note that $w\leq w_0$, hence, by induction we obtain $w_{n}\leq w_{n+1}$. Using the contraction principle we obtain $\lim_{n\to\infty}\| w_{n}-u\|_{C[0,T]}=0$. Since $w_n\leq w_{n+1}$ for every $n=0,1,\ldots$, we obtain $w(t)\leq w_0(t)\leq u(t)$ for $t\in [0,T]$. From \eqref{u-bound} we obtain the desired inequality.

\subsection{The proof of Lemma \ref{Q-lemma}.}
\label{Q-lemma-app}

We first prove (i). We have
\begin{eqnarray*}
   \dnorm{E_\alpha(-t^\alpha A^\beta)\zeta}^2_{s+r}
& = &
   \int_\theta^\infty \lambda^{2(s+r)}
   E_\alpha(-t^\alpha \lambda^\beta) \ud \dnorm{S_\lambda\zeta}^2
\\
&\le &
  \int_\theta^\infty \lambda^{2(s+r)}
  \ud\dnorm{S_\lambda\zeta}^2
=
   \|\zeta\|^2_{s+r}
\end{eqnarray*}
Now, we consider Part (ii). We have $E_{\alpha,\alpha}(z)\ge 0$ (see \cite{Gorenflo-Mittag}, Chap. 3).
 Hence, Lemma \ref{sakamoto} yields
\begin{equation}
\label{es-E-aa}
   \int_{t_1}^{t_2} |\E_{\alpha, \beta}(\lambda, t_2, \tau)| \ud \tau
=
   \int_{t_1}^{t_2} \E_{\alpha, \beta}(\lambda, t_2, \tau) \ud \tau
=
   \frac{1}{\lambda^\beta} H_0(\lambda,t_1,t_2)
\end{equation}
Hence, by the Holder inequality, Lemma \ref{trong1} and \eqref{spectral-representation}, we obtain for $t\in (0,T]$
\begin{eqnarray*}
   \dnorm{Q_{ \alpha, \beta, A}(w)( t_1,t_2 )}_{s+r}^2
&=&
     \dnorm
     {
        \int_{t_1}^{t_2}  |\E_{\alpha, \beta}(A, t_2, \tau)| w(\tau) \ud \tau
     }_{s+r}^2 \nonumber
\\
&\le &
         \int_\theta^{+\infty}
            \lambda^{2(s+r)}
            \int_{t_1}^{t_2}
               |\E_{\alpha, \beta}(\lambda, t_2, \tau)|
            \ud \tau
            \times
            \int_{t_1}^{t_2}
                |\E_{\alpha, \beta}(\lambda, t_2, \tau)|
                \ud\dnorm{\es w(\tau)}^2
            \ud \tau\nonumber
\\
&\le &
    \int_{t_1}^{t_2}
      \int_\theta^{+\infty}
         \frac{\lambda^{2s}}{\lambda^\beta} H_0(\lambda,t_1,t_2)
         |\E_{\alpha, \beta}(\lambda, t_2, \tau)|
      \lambda^{2r}
      \ud\dnorm{\es w(\tau)}^2
   \ud \tau\nonumber
\\
&< &
  \frac{1}{\Gamma(\alpha)}
  \sup_{\lambda\geq\theta} \lambda^{2s-\beta} H_0(\lambda,t_1,t_2)
   \int_{t_1}^{t_2}
   (t_2 -\tau)^{\alpha-1}  \dnorm{w(\tau)}_r^2
   \ud \tau.
\end{eqnarray*}
For  $0<t_1<t_2$, we have
\begin{eqnarray}
   \dnorm{R_{ \alpha, \beta, A}(w)( t_1, t_2 )}_{s+r}^2
&=&
     \dnorm
     {
        \int_0^{t_1}  |\E_{\alpha, \beta}(A, t_1, \tau)-\E_{\alpha, \beta}(A, t_2, \tau)| w(\tau) \ud \tau
     }_{s+r}^2 \nonumber
\\
&\le &
         \int_\theta^{+\infty}
            \lambda^{2(s+r)}
            \int_0^{t_1}
               |\E_{\alpha, \beta}(\lambda, t_1, \tau)-\E_{\alpha, \beta}(\lambda, t_2, \tau)|
            \ud \tau   \nonumber\\
& &      \times
            \int_0^{t_1}
                |\E_{\alpha, \beta}(\lambda, t_1, \tau)-\E_{\alpha, \beta}(\lambda, t_2, \tau)|
                \ud\dnorm{\es w(\tau)}^2
            \ud \tau\nonumber
\end{eqnarray}
From the complete monotonicity of the Mittag-Leffler function $E_{\alpha}(-z)$ for $z\geq 0$ (see \cite{Gorenflo-Mittag}, Chap. 3)
we have  $E_{\alpha,\alpha}(-z)\geq 0$ is decreasing and  $0\leq E_{\alpha}(-z)\leq 1$ for $z\geq 0$. Hence,  we obtain
\begin{eqnarray*}
   |\E_{\alpha, \beta}(\lambda, t_1, \tau)-\E_{\alpha, \beta}(\lambda, t_2, \tau)|
&=&
   -\E_{\alpha, \beta}(\lambda, t_1, \tau)
+
   \E_{\alpha, \beta}(\lambda, t_2, \tau)\\
&=&
   \frac{1}{\lambda^\beta}
   \frac{d}{d\tau}
   (
      E_\alpha(-\lambda^\beta(t_1-\tau)^\alpha)
      -
      E_\alpha(-\lambda^\beta(t_2-\tau)^\alpha)
   ).
\end{eqnarray*}
It follows that
\begin{eqnarray*}
   \int_0^{t_1}
      |\E_{\alpha, \beta}(\lambda, t_1, \tau)-\E_{\alpha, \beta}(\lambda, t_2, \tau)|
   \ud\tau
&=&
   \frac{1}{\lambda^\beta}H(\lambda,t_1,t_2)
\end{eqnarray*}
Hence
\begin{eqnarray*}
\dnorm{R_{ \alpha, \beta, A}(w)( t_1, t_2 )}_{s+r}^2 &\le &
    \int_0^{t_1}
      \int_\theta^{+\infty}
         \frac{\lambda^{2s}}{\lambda^\beta}
         |\E_{\alpha, \beta}(\lambda, t_1, \tau)-\E_{\alpha, \beta}(\lambda, t_2, \tau)|
      \lambda^{2r}
      \ud\dnorm{\es w(\tau)}^2
   \ud \tau\nonumber
\\
&\le &
   \frac{1}{\Gamma(\alpha)}
   \sup_{\lambda \ge \theta}\lambda^{2s-\beta}H(\lambda,t_1,t_2)
   \int_0^{t_1}
   (t_1 -\tau)^{\alpha-1}  \dnorm{w(\tau)}_r^2
   \ud \tau.
\end{eqnarray*}

\subsection{The proof of Lemma \ref{Step1}.}
\label{app-step1}

We can use (\ref{Q-es}) and by direct computations yield
\begin{align*}
\lefteqn
{
\dnorm
{
\F_{\xi, \wt{\alpha}, \wt{\beta}, A}(w) (t)
-
\F_{\zeta, \wt{\alpha}, \wt{\beta}, A}(v) (t)
}^2_s
}
\\
&\le
2 \dnorm
{
   E_{\wt{\alpha}} \big(- A^{\wt{\beta}} t^{\wt{\alpha}} \big)
   (\xi-\zeta)
}_s^2
+
2 \dnorm
{
    Q_{\wt{\alpha}, \wt{\beta}, A}(w)(0, t)
    -
    Q_{\wt{\alpha}, \wt{\beta}, A}(v)(0, t)
}_s^2
\\
& \le
2 \dnorm{\zeta - \xi}_s^2
+
\frac{2}{\Gamma(\wt{\alpha})}\theta^{2s-\wt{\beta}}
\int_0^t
   (t-\tau)^{\wt{\alpha} - 1}
   \dnorm{f( \tau, w) - f( \tau, v)}^2
\ud \tau
\\
& \le
2 \dnorm{\zeta - \xi}_s^2
+
\frac{2}{\Gamma(\wt{\alpha})} \theta^{2s-\wt{\beta}} L^2(\m)
\int_0^t
   (t-\tau)^{\wt{\alpha} - 1} \tau^{-2\nu}
   \dnorm{w(\tau) - v(\tau)}_s^2
\ud \tau.
\end{align*}
Since
$
u_{\xi, \wt{\alpha}, \wt{\beta}}$ and  $u_{\zeta, \wt{\alpha}, \wt{\beta}}
$
are solution of equations
$
\F_{\xi, \wt{\alpha}, \wt{\beta}, A}(w) = w
$
and
$
\F_{\xi, \wt{\alpha}, \wt{\beta}, A}(v) = v
$,
respectively, by Lemma \ref{gronwall}, we conclude that
\begin{eqnarray*}
   \dnorm{
      u_{\xi, \wt{\alpha}, \wt{\beta}}(t)
     -
      u_{\zeta, \wt{\alpha}, \wt{\beta}}(t)
   }_s^2
\le
   2 \Gamma(1-2\nu)
   E_{\wt{\alpha}-2\nu, 1-2\nu}
      \left( 2 \theta^{2s-\wt{\beta}} L^2(\m) t^{\wt{\alpha}-2\nu}\right)
   \dnorm{\zeta - \xi}_s^2.
\end{eqnarray*}
This leads to the result of Lemma \ref{Step1}.

\subsection{The proof of Lemma \ref{Step2}.}
\label{app-step2}

By direct computation, we have
\begin{eqnarray}
\label{q-al-al}
\dnorm
{
   \F_{\zeta, \wt{\alpha}, \wt{\beta}, A}(v)(t)
   -
   \F_{\zeta, \alpha, \beta, A}(u)(t)
}_s^2
&\le &
2 \dnorm
{
  \left(
     E_{\wt{\alpha}} \big(- A^{\wt{\beta}} t^{\wt{\alpha}} \big)
     -
     E_\alpha \big(- A^\beta t^\alpha \big)
  \right) \zeta
}_s^2
 \nn
 \\
 & & \, +
2 \dnorm
{
     Q_{\wt{\alpha}, \wt{\beta}, A}(v)(0, t)
     -
     Q_{\alpha, \beta, A}(u)(0, t)
}_s^2
\nn
\\
& \le &
2  I_1 + 4(I_2+I_3),
\end{eqnarray}
where
\begin{eqnarray*}
I_1
& =  &
\dnorm
{
  \left(
     E_{\wt{\alpha}} \big(- A^{\wt{\beta}} t^{\wt{\alpha}} \big)
     -
     E_\alpha \big(- A^\beta t^\alpha \big)
  \right) \zeta
}_s^2
,
\\
I_2
&= &
\dnorm
{
     Q_{\wt{\alpha}, \wt{\beta}, A}(v)(0, t)
     -
     Q_{\wt{\alpha}, \wt{\beta}, A}(u)(0, t)
}_s^2
,
\\
I_3
& = &
\dnorm
{
     Q_{\wt{\alpha}, \wt{\beta}, A}(u)(0, t)
     -
     Q_{\alpha, \beta, A}(u)(0, t)
}_s^2
  ,
\end{eqnarray*}
and function $Q$ defined in (\ref{Q-define}).
%
%
We will estimate $I_k (k=1, 2, 3)$ one by one.
\\
\textit{ \bf Estimate for  $I_1$}.
To give an estimation for $I_1$, we separate the sum $I_1$ into two sum as follows
\begin{equation}
\label{i1}
   I_1 = I_{11}(N) + I_{12}(N),
\end{equation}
where
\begin{eqnarray*}
I_{11}( N)
&=&
\int_\theta^N
  \lambda^{2s}
  \left|
      E_{\wt{\alpha}} \big(- \lambda^{\wt{\beta}} t^{\wt{\alpha}} \big)
      -
       E_\alpha \big(- \lambda^\beta t^\alpha \big)
  \right|^2
\ud \dnorm{\es \zeta}^2,
 \\
I_{12}(N)
&=&
\int_{\lambda>N}
  \lambda^{2s}
  \left|
      E_{\wt{\alpha}} \big(- \lambda^{\wt{\beta}} t^{\wt{\alpha}} \big)
      -
       E_\alpha \big(- \lambda^\beta t^\alpha \big)
  \right|^2
\ud \dnorm{\es \zeta}^2.
\end{eqnarray*}
For convenience in estimating for $I_{11}(N), I_{12}(N)$,
let us assume $N> \max\{e, \theta\}$.
\\
\textit{\bf Estimate for $I_{11}(N)$.} By Lemma \ref{trong1},
there exist two constants
$
 C=C(\alpha_*, \alpha^*, \beta_*, \beta^*, T)>0
$
,
$
 C_0=C_0(\alpha_*, \alpha^*, \beta_*, \beta^*, \theta, T)>0
$
such that
\begin{eqnarray}
\label{i11}
I_{11}(N)
& \le &
C \left( |\alpha - \wt{\alpha}| + |\beta - \wt{\beta}|\right)^2
\left(
   \int_\theta^N \lambda^{2(\beta^* + s)}(1+ |\ln \lambda |)^2 \ud \dnorm{\es \zeta}^2
\right)
\nn
\\
&\le &
C_0 \left( |\alpha - \wt{\alpha}| + |\beta - \wt{\beta}|\right)^2
N^{2\beta^*}  \ln^2 N
  \int_\theta^N \lambda^{2s}\ud \dnorm{\es \zeta}^2
\nn
\\
&\le &
C_N  \left( |\alpha - \wt{\alpha}| + |\beta - \wt{\beta}|\right)^2,
\end{eqnarray}
where $C_N=C_0\dnorm{\zeta}^2_s N^{2\beta^*} \ln^2 N$.
~
\\
~
\textit{\bf Estimate for $I_{12}(N)$.}
We note that $0\le E_\alpha(-x) \le 1$ for $x>0$. This gives
\begin{equation}
\label{i12}
I_{12}(N)
\le
\int_{\lambda>N} \lambda^{2s} \ud \dnorm{\es \zeta}^2.
\end{equation}
Substituting (\ref{i11}) and (\ref{i12}) into (\ref{i1}), we obtain
\begin{equation}
\label{es-i1}
I_1
\le
C_N (|\alpha - \wt{\alpha} | + | \beta - \wt{\beta} |)^2
   +
\int_{\lambda>N} \lambda^{2s} \ud \dnorm{\es \zeta}^2,
\end{equation}
where $C_N$ defined in (\ref{i11}).
%
%
\\
\\
\textit{ \bf Estimate for  $I_2$}.
Similarly to the proof of Lemma \ref{Step1}, we get
\begin{equation}\label{i2}
   I_2
\le
   \frac{1}{\Gamma(\wt{\alpha})}\theta^{2s-\wt{\beta}} L^2(\m)
   \int_0^t
   (t - \tau)^{\wt{\alpha} - 1} \tau^{-2\nu}\dnorm{v(\tau) - u(\tau)}_s^2
\ud \tau.
\end{equation}
%
%
\textit{\bf Estimate for  $I_3$}.
Recall the definition $Q$ which defined in (\ref{Q-define}) as follows
\[
Q_{\alpha, \beta, A}(u)(0, t)
 =
\int_0^t
  \E_{\alpha, \beta}(A, t, \tau) f(\tau, u)
\ud \tau
.
\]
By the Holder inequality and direct computation, we have
\begin{eqnarray}\label{i3}
I_3
&\le &
\int_\theta^{+\infty}
\lambda^{2s}
\int_0^t
   \left|
      \E_{\alpha, \beta}(\lambda, t, \tau)
      -
      \E_{\wt{\alpha}, \wt{\beta}}(\lambda, t, \tau)
   \right|
\ud\tau
\times
\nn
\\
&&
\times
\int_0^t
   \left|
      \E_{\alpha, \beta}(\lambda, t, \tau)
      -
      \E_{\wt{\alpha}, \wt{\beta}}(\lambda, t, \tau)
   \right|
  \ud \dnorm{\es f(\tau, u(\tau))}^2
\ud \tau
\nn
\\
&=&
I_{31}(N)+I_{32}(N),
\end{eqnarray}
where
\begin{eqnarray}
I_{31}(N)
 &=&
\int_\theta^N
\lambda^{2s}
\int_0^t
   \left|
      \E_{\alpha, \beta}(\lambda, t, \tau)
      -
      \E_{\wt{\alpha}, \wt{\beta}}(\lambda, t, \tau)
   \right|
\ud \tau
\times
\nn
\\
\label{df-i31}
&&
\times
\int_0^t
   \left|
      \E_{\alpha, \beta}(\lambda, t, \tau)
      -
      \E_{\wt{\alpha}, \wt{\beta}}(\lambda, t, \tau)
   \right|
  \ud \dnorm{\es f(\tau, u(\tau))}^2
\ud \tau ,
\nn
\\
I_{32}(N)
& =&
\int_{\lambda> N}
\lambda^{2s}
\int_0^t
   \left|
      \E_{\alpha, \beta}(\lambda, t, \tau)
      -
      \E_{\wt{\alpha}, \wt{\beta}}(\lambda, t, \tau)
   \right|
\ud\tau
\times
\nn
\\
&&
\times
\int_0^t
   \left|
      \E_{\alpha, \beta}(\lambda, t, \tau)
      -
      \E_{\wt{\alpha}, \wt{\beta}}(\lambda, t, \tau)
   \right|
  \ud \dnorm{\es f(\tau, u(\tau))}^2
\ud \tau.
\end{eqnarray}
We will estimate for $ I_{31}(N)$ and $ I_{32}(N)$ one by one.
\\
\textit{\bf Estimate for $ I_{31}(N)$.}
By Lemma \ref{trong1}, we have
\begin{equation*}
\int_0^t
   \left|
      \E_{\alpha, \beta}(\lambda, t, \tau)
      -
      \E_{\wt{\alpha}, \wt{\beta}}(\lambda, t, \tau)
   \right|
\ud\tau
\le
C_1
  \left(
    (1+\lambda^\beta) |\alpha - \wt{\alpha}|
    +
     |\lambda^\beta - \lambda^{\wt{\beta}} |
  \right).
\end{equation*}
By the mean value theorem, for $\lambda \le N$ with $N$ large enough, we obtain
\begin{eqnarray}
\label{i31-1}
\int_0^t
   \left|
      \E_{\alpha, \beta}(\lambda, t, \tau)
      -
      \E_{\wt{\alpha}, \wt{\beta}}(\lambda, t, \tau)
   \right|
\ud\tau
&\le &
C_2 \lambda^{\beta^*} |\ln \lambda|
\left( |\alpha - \wt{\alpha}|+ |\beta-\wt{\beta}| \right)
.
\end{eqnarray}
On the other hand, there exists $C_3=C_3(\alpha_*, \alpha^*, \beta_*)$ such that
\begin{eqnarray}
\label{i31-2}
\left|
  \E_{\alpha, \beta}(\lambda_k, t, \tau)
  -
  \E_{\wt{\alpha}, \wt{\beta}}(\lambda_k, t, \tau)
\right|
& \le &
C_3 \left( (t-\tau)^{\alpha-1}+(t-\tau)^{\wt{\alpha}-1}\right)
\nn
\\
& \le &
2 C_3 \left( (t-\tau)^{\alpha_*-1}+(t-\tau)^{\alpha^*-1}\right)
.
\end{eqnarray}
Plugging (\ref{i31-1}) and (\ref{i31-2}) into (\ref{df-i31}), we get that
\begin{eqnarray}
\label{i31-es1}
   I_{31}(N)
 &\le &
   C_4 N^{\beta^*+2s} \ln N
   \left( |\alpha - \wt{\alpha}|+ |\beta-\wt{\beta}| \right)
\nn
\\
&\times &
   \int_\theta^N
      \int_0^t
        \left( (t-\tau)^{\alpha^*-1}+(t-\tau)^{\alpha_*-1}\right)
        \ud \dnorm{\es f(\tau, u(\tau))}^2
     \ud \tau
\end{eqnarray}
for $N$ large enough
and $C_4=2C_2C_3$.
~
~
Furthermore, thank to the condition (\ref{local-lips}), we get that
\begin{eqnarray}
\label{sum-a-es}
\lefteqn
{
   \int_\theta^N
      \int_0^t
        \left( (t-\tau)^{\alpha^*-1}+(t-\tau)^{\alpha_*-1}\right)
        \ud \dnorm{\es f(\tau, u(\tau))}^2
     \ud \tau
}
\nn
\\
&\le &
\int_\theta^{+\infty}
\int_0^t
  \left( (t-\tau)^{\alpha^*-1}+(t-\tau)^{\alpha_*-1}\right) \ud \dnorm{\es f(\tau, u(\tau))}^2
\ud \tau
\nn
\\
&\le &
\int_0^t
  \left( (t-\tau)^{\alpha^*-1}+(t-\tau)^{\alpha_*-1}\right)
  \left(\dnorm{f(\tau, 0)}^2 + L^2(\m) \m^2  t^{-2\nu} \right)
\ud \tau
\nn
\\
&:=&
C_5,
\end{eqnarray}
where $C_5=C_5(\alpha_*, \alpha^*, \beta_*, \m)$.
Combining the inequality (\ref{i31-es1}) with (\ref{sum-a-es}),
we obtain
\begin{equation}
 \label{i31}
   I_{31}(N)
\le
   D_N
   \left(|\alpha - \wt{\alpha}| + |\beta - \wt{\beta}|\right),
 \end{equation}
where $D_N=C_4C_5 N^{\beta^*+2s} \ln N$.
\\
\textit{\bf Estimate  for $I_{32}(N)$.}
Thanks to (\ref{es-E-aa}), one has
\begin{eqnarray*}
\int_0^t
\left|
  \E_{\alpha, \beta}(\lambda, t, \tau)
 -
  \E_{\wt{\alpha}, \wt{\beta}}(\lambda, t, \tau)
\right|
\ud \tau
& \le &
\int_0^t
  \E_{\alpha, \beta}(\lambda, t, \tau)
\ud \tau
  +
\int_0^t
  \E_{\wt{\alpha}, \wt{\beta}}(\lambda, t, \tau)
\ud \tau
\\
&\le &
1/\lambda^\beta + 1/\lambda^{\wt{\beta}}.
\end{eqnarray*}
Consequently,
\[
\lambda^s \int_0^t
\left|
  \E_{\alpha, \beta}(\lambda, t, \tau)
 -
  \E_{\wt{\alpha}, \wt{\beta}}(\lambda, t, \tau)
\right|
\le C_6,
\]
where
$C_6=C_6(\beta_*, \beta^*, \theta)$,
and that
\begin{equation}
\label{i32}
I_{32}(N)
\le
C_6 \int_{\lambda>N}
\int_0^t
\left|
  \E_{\alpha, \beta}(\lambda, t, \tau)
 -
  \E_{\wt{\alpha}, \wt{\beta}}(\lambda, t, \tau)
\right|
\ud \dnorm{\es f(\tau, u(\tau))}^2
\ud\tau
\end{equation}
From (\ref{es-i1}), (\ref{i3}), (\ref{i31}) and (\ref{i32}),
for $|\alpha - \wt{\alpha} | + | \beta - \wt{\beta} |\le 1$,
we obtain
\begin{eqnarray}
\label{i1+i3}
   I_1 + I_3
&\le &
   E_N \left(|\alpha - \wt{\alpha} | + | \beta - \wt{\beta} | \right)
  +
   2 \int_{\lambda>N} \lambda^{2s} \ud \dnorm{ \es \zeta}^2
\nn
\\
&+&
   4 C_6
   \int_{\lambda>N}
   \int_0^t
   \left|
        \E_{\alpha, \beta}(\lambda, t, \tau)
       -
       \E_{\wt{\alpha}, \wt{\beta}}(\lambda, t, \tau)
    \right|
   \ud \dnorm{\es f(\tau, u(\tau))}^2
   \ud\tau
   ,
\end{eqnarray}
where $E_N=2C_N+6D_N$
with $C_N$ defined in (\ref{i11}) and $D_N$ defined in (\ref{i31}).

Let us mention (\ref{sum-a-es}) that
\begin{equation*}
\int_\theta^{+\infty}
\int_0^t
  \left( (t-\tau)^{\alpha^*-1}+(t-\tau)^{\alpha_*-1}\right) \ud \dnorm{\es f(\tau, u(\tau))}^2
\ud \tau
\le
C_5
\end{equation*}
and $\zeta \in D(A^s)$.
This leads to the fact that there exists $N=N(\epsilon)$
independent of $\alpha, \wt{\alpha}$ and $\beta, \wt{\beta}$ such that
\[
   2 \int_{\lambda>N} \lambda^{2s} \ud \dnorm{ \es \zeta}^2
+
   4 C_6
   \int_{\lambda>N}
   \int_0^t
   \left|
        \E_{\alpha, \beta}(\lambda, t, \tau)
       -
       \E_{\wt{\alpha}, \wt{\beta}}(\lambda, t, \tau)
    \right|
   \ud \dnorm{\es f(\tau, u(\tau))}^2
   \ud\tau
 <
 \epsilon.
\]
By (\ref{i1+i3}), we obtain the following estimate
\begin{equation}
\label{i-1-3}
I_1+ I_3
\le
\epsilon
+
P_\epsilon (|\alpha -\wt{\alpha}| + |\beta - \wt{\beta}|).
\end{equation}
Substituting (\ref{i2}) and (\ref{i-1-3}) into (\ref{q-al-al}), we obtain
\begin{eqnarray*}
\lefteqn
{
\dnorm{
  \F_{\zeta, \wt{\alpha}, \wt{\beta}, A}(v)(t)
   -
  \F_{\zeta, \alpha, \beta, A}(u)(t)
}_s^2
}
 \\
& \le &
4\left(\epsilon + P_\epsilon(|\alpha -\wt{\alpha}| + |\beta - \wt{\beta}|)\right)
+
\frac{4}{\Gamma(\wt{\alpha})}\theta^{2s-\wt{\beta}} L^2(\m)
\int_0^t
  (t - \tau)^{\wt{\alpha} -1} \tau^{-2\nu} \dnorm{v(\tau)-u(\tau)}^2_s
\ud \tau.
\end{eqnarray*}
Since $u_{\zeta, \wt{\alpha}, \wt{\beta}}$ and $u_{\zeta, \alpha, \beta}$ are the solution of
the equations
$\F_{\zeta, \wt{\alpha}, \wt{\beta}, A}(v)=v$ and $\F_{\zeta, \alpha, \beta, A}(u)=u$,
respectively.
By Lemma \ref{gronwall}, we conclude that
\begin{eqnarray*}
\dnorm
  {
    u_{\zeta, \wt{\alpha}, \wt{\beta}}(t)
   -
    u_{\zeta, \alpha, \beta}(t)
   }^2_s
\le
   P_0 \left(\epsilon+P_\epsilon(|\alpha -\wt{\alpha}| + |\beta - \wt{\beta}|)\right)
   E_{\wt{\alpha}-2\nu, 1-2\nu}
   \left(4L^2(\m) \theta^{2s-\wt{\beta}} t^{\wt{\alpha}-2\nu}\right),
\end{eqnarray*}
where $P_0=4 \Gamma(1-2\nu)$.
This completes the proof of Lemma \ref{Step2}.

\subsection{The proof of Lemma \ref{Step3}.}
\label{app-step3}

Analogously, the proof of Lemma \ref{Step2},
we can use Lemma \ref{gronwall} to prove that
\begin{equation*}
\dnorm
{
  u_{\zeta, \wt{\alpha}, \wt{\beta}}(t)
  -
  u_{\zeta, \alpha, \beta}(t)
}_s
\le
Q_0
\left(
   N^{-\gamma_2}
    +
   N^{\gamma_1} \ln^2 N
   \left(|\alpha - \wt{\alpha}| + |\beta - \wt{\beta} |\right)
\right)^{1/2},
\end{equation*}
where $Q_0$ is independent of $N, \, \alpha , \, \wt{\alpha}, \, \beta, \, \wt{\beta}$.
~
\\
~
Since $\ln N<N$, we obtain
\begin{equation}
\label{uv-es}
\dnorm
{
  u_{\zeta, \wt{\alpha}, \wt{\beta}}(t)
  -
  u_{\zeta, \alpha, \beta}(t)
}_s
\le
Q_0
\left(
  N^{-\gamma_2}
  +
  N^{\gamma_1+2}
  \left(|\alpha - \wt{\alpha}| + |\beta - \wt{\beta} |\right)
\right)^{1/2},
\end{equation}
Let us suppose that $|\alpha - \wt{\alpha}| + |\beta - \wt{\beta} | \le 1$,
and we can choose
$
N
=
\left[
  (|\alpha - \wt{\alpha}| + |\beta - \wt{\beta} |)^{- 1/(2(\gamma_1+\gamma_2+1))}
\right] +1
$.
It is easy to see that
$
(|\alpha - \wt{\alpha}| + |\beta - \wt{\beta} |)^{- 1/(\gamma_1+\gamma_2++2)}
<
N
\le
2 (|\alpha - \wt{\alpha}| + |\beta - \wt{\beta} |)^{-1/(\gamma_1+\gamma_2+2)}
$.
Hence, by (\ref{uv-es}), we obtain
\begin{equation*}
\dnorm
{
  u_{\zeta, \wt{\alpha}, \wt{\beta}}(t)
  -
  u_{\zeta, \alpha, \beta}(t)
}_s
\le
Q_0 \left( 2^{\gamma_1+2} +1 \right)
\left(
    |\alpha - \wt{\alpha}| + |\beta - \wt{\beta} |
\right)^{\gamma_2/\left(2(\gamma_1+\gamma_2+2)\right)}.
\end{equation*}
~
~
This completed the proof of Lemma \ref{Step3}.

\subsection{The proof of Lemma \ref{QG-es}.}
\label{app-QG-es}

{\bf Proof of part (1).}
By Lemma \ref{Q-lemma}, we have
\begin{eqnarray*}
Q_{\alpha, \beta, A}(w_1)(t) -Q_{\alpha, \beta, A}(w_2)(t)
=
\int_0^t
  \E_{\alpha, \beta}(A, t, \tau)
  \left(f(\tau, w_1) -f(\tau, w_2)\right)
\ud \tau
.
\end{eqnarray*}
Using (\ref{Q-es}), one has
\begin{eqnarray}
\label{H-es1-nonlinear}
\| Q_{\alpha, \beta, A}(w_1)(t) -Q_{\alpha, \beta, A}(w_2)(t)\|_s^2
& \le &
\frac{\theta^{2s-\beta}}{\Gamma(\alpha)}
\int_0^t
  (t-\tau)^{\alpha - 1} \tau^{-2\nu}
  \| w_1(\tau)-w_2(\tau)\|^2
\ud \tau.
\nn
\\
&\le &
\kappa^2\frac{\theta^{2s-\beta}}{\Gamma(\alpha)} B(\alpha, 1-2\rho-2\nu)\tnorm{w_1-w_2)}_{s, \rho}^2
t^{\alpha-2\rho-2\nu}
\nn
\\
&=&
\kappa^2 \theta^{2s-\beta} E_0
\tnorm{w_1-w_2)}_{s, \rho}^2
t^{\alpha-2\rho-2\nu}
,
\end{eqnarray}
where $E_0=B(\alpha, 1-2\rho-2\nu)/\Gamma(\alpha)$.
The latter inequality deduces the result of part (1).
\\
~
~
{\bf Proof of part (2).}
By triangle inequality, we have
\begin{eqnarray*}
\dnorm
{
 G_{\varphi, \alpha, \beta, A} (w_1) - G_{\wt{\varphi}, \alpha, \beta, A} (w_2)
}_s
 &\le &
  \dnorm{\varphi - \wt{\varphi}}_s
   +
  \dnorm{Q_{\alpha, \beta, A}(w_1)(T) - Q_{\alpha, \beta, A}(w_2)(T)}_s
  .
\end{eqnarray*}
Hence, thank to (\ref{H-es1-nonlinear}), we obtain the result of part (2).
~
\\
~
{\bf Proof of part (3).}
Let us mention the Lemma \ref{Q-lemma} that
\begin{equation*}
\dnorm
{
 Q_{\alpha, \beta, A}(w_1)(t)
}_s^2
\le
\frac{\theta^{2s-\beta}}{\Gamma(\alpha)}
\int_0^t  (t-\tau)^{\alpha - 1} \| f(\tau, w_1)\|^2\ud \tau
,
\end{equation*}
By {\bf Assumption F1}, we have
\begin{eqnarray*}
\dnorm{f(t, w_1)}
\le
\kappa t^{-\nu} \dnorm{w_1(t)}_s + \dnorm{f(t, 0)}
.
\end{eqnarray*}
Thus,
\begin{eqnarray*}
\| Q_{\alpha, \beta, A}(w_1)(t) \|_s^2
&\le &
\frac{2\theta^{2s-\beta}}{\Gamma(\alpha)}
\int_0^t
  (t-\tau)^{\alpha - 1}
  \left( \tau^{-2\nu} \kappa^2 \dnorm{w_1(\tau)}_s^2
  +
  \dnorm{f(\tau, 0)}^2\right)
\ud \tau
\\
& \le &
2\theta^{2s-\beta}
\left(
\frac{\Theta_\alpha(t)}{\Gamma(\alpha)}
+
\kappa^2 E_0
\tnorm{w_1}_{s,\rho}^2 t^{\alpha-2\rho-2\nu}
\right)
,
\end{eqnarray*}
where
$
\Theta_\alpha(t)
=
\int_0^t (t-\tau)^{\alpha - 1} \dnorm{f(\tau, 0)}^2 \ud \tau
$.
This completes the proof of part (3).
~
\\
~
{\bf Proof of part (4).}
We have
\begin{equation*}
\dnorm{G_{\varphi, \alpha, \beta, A}(w_1)}_s^2
\le
2
\left(
   \| \varphi\|_s^2
   +
   \dnorm{Q_{\alpha, \beta, A}(w_1)(T)}_s^2
\right)
    .
\end{equation*}
Hence, using the result of part (3), we obtain the desired result of part (4).
~
\\
~
{\bf Proof of part (5).}
Since
\begin{equation*}
\dnorm{G_{\varphi, \alpha, \beta, A}(w_1)}_{s+r}^2
\le
2
\left(
   \| \varphi\|_{s+r}^2
   +
   \dnorm{Q_{\alpha, \beta, A}(w_1)(T)}_{s+r}^2
\right)
    .
\end{equation*}
Then by Lemma \ref{Q-lemma} and the assumption $\varphi \in D\left(A^{s+r} \right)$
we obtain the results of part (5).
~
~
This completes the proof of the Lemma \ref{QG-es}.

\subsection{The proof of Lemma \ref{Step1-2}.}
\label{app-Step1-2}

Put
$
\gamma
=
\min\{\beta/2, \, \wt{\beta}/2\}
$.
We observe that
\begin{equation*}
\cq_{\wt{\varphi}, \wt{\alpha}, \wt{\beta}, A}(w)(t)
-
\cq_{\varphi, \wt{\alpha}, \wt{\beta}, A}(v)(t)
=
\left[
   \cq_{\wt{\varphi}, \wt{\alpha}, \wt{\beta}, A}(w)(t)
   -
   \cq_{\wt{\varphi}, \wt{\alpha}, \wt{\beta}, A}(v)(t)
\right]
+
P_{\alpha, \beta}(A, t) \left(\wt{\varphi}-\varphi \right)
.
\end{equation*}
Similarly (\ref{Q-w12}), we have
\begin{eqnarray}
\label{Q-w-v}
\tnorm
{
   \cq_{\wt{\varphi}, \wt{\alpha}, \wt{\beta}, A}(w)
    -
   \cq_{\varphi, \wt{\alpha}, \wt{\beta}, A}(v)
}_{\gamma, \, \rho}
\le
\kappa/K_m \tnorm{w - v}_{\gamma, \, \rho}
+
\tnorm
{
P_{\alpha, \beta}(A, t) \left(\wt{\varphi}-\varphi \right)
}_{\gamma, \, \rho}.
\end{eqnarray}
By (\ref{P-es}), we can estimate for second term in the right--hand side as follow
\begin{eqnarray*}
\tnorm
{
P_{\alpha, \beta}(A, t) \left(\wt{\varphi}-\varphi \right)
}^2_{\gamma, \, \rho}
&\le &
E^2 T^{2\rho}
\int_\theta^{+\infty}
  \lambda^{2 \gamma}
  \ud \dnorm{\es \left(\wt{\varphi}-\varphi\right)}^2
\\
& \le &
E^2 T^{2\rho} \max\left\{1, \, \theta^{2\gamma-2r-\beta^*}\right\}
\int_\theta^{+\infty}
  \lambda^{\beta^*+2r}
  \ud \dnorm{\es \left(\wt{\varphi}-\varphi\right)}^2
\\
&=&
E^2 T^{2\rho} \max\left\{1, \, \theta^{2\gamma-2r-\beta^*}\right\}
\| \wt{\varphi} - \varphi\|^2_{\beta^*/2+r}.
\end{eqnarray*}
Since
$
u_{\wt{\varphi}, \wt{\alpha}, \wt{\beta}},
u_{\varphi, \wt{\alpha}, \wt{\beta}}
$
are solutions
of the equations
$
\cq_{\wt{\varphi}, \wt{\alpha}, \wt{\beta}, A}(w)=w,
\,
\cq_{\varphi, \wt{\alpha}, \wt{\beta}, A}(v)=v
$,
respectively.
Substituting the last inequality into (\ref{Q-w-v}), we obtain
\[
\tnorm
{
   u_{\wt{\varphi}, \wt{\alpha}, \wt{\beta}}
   -
   u_{\varphi, \wt{\alpha}, \wt{\beta}}
}_{\gamma, \, \rho}
 \le
\left(1-\kappa/K_m\right)^{-1}
E T^{\rho} \max\left\{1, \, \theta^{\gamma-r-\beta^*/2}\right\}
\| \wt{\varphi} - \varphi\|_{\beta^*/2+r}
.
\]
This completed the proof of Lemma \ref{Step1-2}.

\subsection{The proof of Lemma \ref{Step2-2}.}
\label{app-Step2-2}

Put $\gamma=\min\{ \beta/2, \wt{\beta}/2\}$.
By direct computation, we have
\begin{eqnarray}
\label{Q-es-al-be}
\lefteqn
{
\tnorm
{
  \cq_{\varphi, \wt{\alpha}, \wt{\beta}, A}(v)
  -
  \cq_{\varphi, \alpha, \beta, A}(u)
}_{\gamma, \, \rho}
}
\nn
\\
& \le &
\tnorm
{
  \cq_{\varphi, \wt{\alpha}, \wt{\beta}, A}(v)
  -
  \cq_{\varphi, \wt{\alpha}, \wt{\beta}, A}(u)
}_{\gamma, \, \rho}
 +
\tnorm
{
  \cq_{\varphi, \wt{\alpha}, \wt{\beta}, A}(u)
  -
  \cq_{\varphi, \alpha, \beta, A}(u)
}_{\gamma, \, \rho}
\nn
 \\
&\le &
\kappa/K_m \tnorm{v - u}_{\gamma, \, \rho}
 +
\tnorm
{
  \cq_{\varphi, \wt{\alpha}, \wt{\beta}, A}(u)
  -
  \cq_{\varphi, \alpha, \beta, A}(u)
}_{\gamma, \, \rho}
.
\end{eqnarray}
To get the desired result, we find an estimation for the second term in the last line.
By the triangle inequality and direct computation, we have
\begin{equation}
\label{Qw-es}
\tnorm
{
  \cq_{\varphi, \wt{\alpha}, \wt{\beta}, A}(u)
  -
  \cq_{\varphi, \alpha, \beta, A}(u)
}_{\gamma, \, \rho}
\le
J_1+J_2+J_3
    ,
\end{equation}
where
\begin{eqnarray*}
J_1
& = &
\tnorm
{
  \left[
     P_{\wt{\alpha}, \wt{\beta}}(A, t)
     -
     P_{\alpha, \beta}(A, t)
   \right]
G_{\varphi, \alpha, \beta, A}(u)
}_{\gamma, \, \rho},
\\
J_2
& = &
\tnorm
{
   P_{\alpha, \beta}(A, t)
    \left[
       Q_{\wt{\alpha}, \wt{\beta}, A}(u)(T)
       -
       Q_{\alpha, \beta, A}(u)(T)
     \right]
}_{\gamma, \, \rho}
\nn
\\
J_3
&=&
\tnorm
{
  \left[
     Q_{\wt{\alpha}, \wt{\beta}, A}(u)(t)
     -
     Q_{\alpha, \beta, A}(u)(t)
   \right]
}_{\gamma, \, \rho}
   ,
\end{eqnarray*}
and function $Q$ defined in (\ref{solution-bw-nonlinear}).

Now we find estimates for $J_k (k=1,2)$.

\textit{\bf Estimate for  $J_1$.}
For $N>\theta$, it is easy to see that
\begin{equation}
\label{J1}
 J_1^2
 \le
 J_{11}(N) + J_{12}(N),
\end{equation}
 where
\begin{eqnarray*}
J_{11}(N)
 & = &
\sup_{t \in (0, T]}
t^{2\rho}
\int_\theta^N
   \lambda^{2\gamma}
   \left|
     P_{\wt{\alpha}, \wt{\beta}}(\lambda, t)
     -
     P_{\alpha, \beta}(\lambda, t)
   \right|^2
  \ud \dnorm{\es G_{\varphi, \alpha, \beta, A}(u)}^2,
\, \,
\\
J_{12}(N)
& = &
\sup_{t \in (0, T]}
t^{2\rho}
\int_{\lambda>N}
  \lambda^{2\gamma}
  \left|
     P_{\wt{\alpha}, \wt{\beta}}(\lambda, t)
     -
     P_{\alpha, \beta}(\lambda, t)
   \right|^2
  \ud \dnorm{\es G_{\varphi, \alpha, \beta, A}(u)}^2.
\end{eqnarray*}
By part (b) of Lemma \ref{trong1}, we can find $C_1, C_2$ dependent only on
$\alpha_*, \alpha^*$ such that
\[
\frac{C_{11}}{\lambda}
\le
E_a(-\lambda)
\le
\frac{C_{12}}{\lambda},
\]
for any $a\in [\alpha_*, \alpha^*]$ and for every $\lambda>0$.
Hence, by part (c) of Lemma \ref{trong1} and (\ref{P-es}),
we have
\begin{eqnarray*}
\left| P_{\wt{\alpha}, \wt{\beta}}(\lambda, t) - P_{\alpha, \beta}(\lambda, t) \right|
&= &
\frac
{
   E_{\alpha} \big(- \lambda^\beta t^\alpha \big)
   \snorm
    {
       E_{\wt{\alpha}} \big(- \lambda^{\wt{\beta}} T^{\wt{\alpha}}\big)
       -
       E_{\alpha} \big(- \lambda^\beta T^\alpha\big)
    }
}
{
   E_{\alpha} \big(- \lambda^\beta T^\alpha \big)
   E_{\wt{\alpha}} \big(- \lambda^{\wt{\beta}} T^{\wt{\alpha}} \big)
}
\\
&+&
\frac
{
   \snorm
   {
       E_{\wt{\alpha}} \big(- \lambda^{\wt{\beta}} t^{\wt{\alpha}}\big)
       -
      E_{\alpha} \big(- \lambda^\beta t^\alpha\big)
    }
}
{
E_{\wt{\alpha}} \big(- \lambda^{\wt{\beta}} T^{\wt{\alpha}} \big)
}
\\
 & \le &
  M
  \left(
     T^\alpha /t^\alpha
     +
     1
  \right)
  \lambda^{2\beta^*} ( \ln \lambda + 1 )
  \left(|\alpha - \wt{\alpha}| + | \beta - \wt{\beta}|\right)
\\
& \le &
  4 M
  \left( T^\alpha/t^\alpha \right)
  N^{2\beta^*} \ln N
  \left(|\alpha - \wt{\alpha}| + | \beta - \wt{\beta}|\right)
\end{eqnarray*}
for any $\lambda \le N$ and for $N$ large enough, where $M=M(\alpha_*, \alpha^*, \beta_*, \beta^*)$.
Since $u \in C_{\gamma, \, \rho}(T)$ and has an upper bound,
we can use part (4) of Lemma \ref{QG-es} and (\ref{F-upper}) to get  the following estimate
\begin{eqnarray}
\label{J11-es}
   J_{11}(N)
&\le &
   16 M^2 T^{2\rho}
   N^{4\beta^*} \ln^2 \lambda_N
   \left(|\alpha - \wt{\alpha}| + | \beta - \wt{\beta}|\right)^2
   \dnorm{G_{\varphi, \alpha, \beta, A}(u)}^2_{\gamma}
\nn
\\
& \le &
   M_1^2
   N^{4\beta^*} \ln^2 N
   \left(|\alpha - \wt{\alpha}| + | \beta - \wt{\beta}|\right)^2
,
\end{eqnarray}
where $M_0$, $M_1$ only depend on $T$, $\alpha_*, \alpha^*$, $\beta_*, \beta^*$, $\theta$. On the other hand, we have
\[
\snorm{P_{\wt{\alpha}, \wt{\beta}}(\lambda, t) - P_{\alpha, \beta}(\lambda, t)}
   \le
\snorm{P_{\wt{\alpha}, \wt{\beta}}(\lambda, t)}
+
\snorm{P_{\alpha, \beta}(\lambda, t)}
\le
M \left( T^\alpha/t^\alpha+ T^{\wt{\alpha}}/t^{\wt{\alpha}} \right)
\le
2 M  T^{\alpha^*}/t^{\alpha^*}
.
\]
Consequently, for $N$ large enough,
we have
\begin{eqnarray}
\label{j12-N}
   J_{12}(N)
&\le &
   4 M^2 T^{2\rho}
   \int_{\lambda>N}
   \lambda^{2\gamma}
   \ud \dnorm{\es G_{\varphi, \alpha, \beta, A}(u)}^2
   .
\end{eqnarray}
Since $G_{\varphi, \alpha, \beta, A}(u) \in D(A^\gamma)$,
and by part (4) of Lemma \ref{QG-es} and (\ref{F-upper}),
there exists $N_1>0$ independent of $\gamma$ such that
\begin{equation}
\label{J12-es}
   J_{12}(N) \le \epsilon^2/4
\end{equation}
for any $N \ge N_1$.
Substituting (\ref{J11-es}) and (\ref{J12-es}) into (\ref{J1}),
we obtain
\begin{equation}
\label{J1-es}
   J_1
\le
   M_1
   N^{2\beta^*} \ln N
   \left(|\alpha - \wt{\alpha}| + | \beta - \wt{\beta}|\right)
+
   \epsilon/2
.
\end{equation}
\textit{\bf Estimate for  $J_2$.}
Applying (\ref{P-es}), we have
$
   P_{\alpha, \beta}(\lambda, t)
\le
   E T^\alpha/t^\alpha
\le
   E T^{\rho}/t^{\rho}
$.
This implies
\begin{eqnarray*}
J_2^2
&=&
\tnorm
{
  P_{\alpha, \beta}(A, t)
  \left[
     Q_{\wt{\alpha}, \wt{\beta}, A}(u)(T)
     -
     Q_{\alpha, \beta, A}(u)(T)
  \right]
}_{\gamma, \, \rho}^2
\nn
\\
&\le &
E^2 T^{2\rho}
  \dnorm
  {
    Q_{\wt{\alpha}, \wt{\beta}, A}(u)(T)
    -
    Q_{\alpha, \beta, A}(u)(T)
  }_{\gamma}^2
\nn
\\
&\le &
E^2 J_3^2.
\end{eqnarray*}
The latter inequality lead to
\[
  J_2+J_3 \le (E+1)J_3.
\]
Therefore, we have to find the estimation for $J_3$.
~
~
By the same method that used to estimate for $I_3$ in the proof of Lemma \ref{app-step2},
we can easy to prove that, there exists $N_2=N_2(\epsilon)$ such that for $N \ge N_2$ such that
\begin{equation*}
J_{3}
\le
D_1 N^{\beta^*} \ln N
\left(
    |\alpha - \wt{\alpha}| +|\beta - \wt{\beta}|
\right)^{1/2}
+
\epsilon/(2(E+1))
,
\end{equation*}
where $D_1$ is independent of $\alpha, \, \wt{\alpha}, \, \beta, \, \wt{\beta}$.
From last two  inequalities, we obtain
\[
J_2+J_3
\le
D_2
N^{\beta^*} \ln N
\left(
    |\alpha - \wt{\alpha}| +|\beta - \wt{\beta}|
\right)
+
\epsilon/2
\]
for any $N \ge N_2$.
Here
$D_2$ is independent of $\alpha, \, \wt{\alpha}, \, \beta, \, \wt{\beta} \, \epsilon$.
We substitute the last inequality and (\ref{J1-es}) into (\ref{Qw-es}),
to  obtain
\[
\tnorm
{
  \cq_{\varphi, \wt{\alpha}, \wt{\beta}, A}(u)
  -
  \cq_{\varphi, \alpha, \beta, A}(u)
}_{\gamma, \, \rho}
\le
\epsilon
+
D N^{2 \beta^*} \ln N
\left(
    |\alpha - \wt{\alpha}| +|\beta - \wt{\beta}|
\right)^{1/2}
.
\]
for any $N \ge \max\{N_1, N_2\}$.
~
Combining the latter inequality with (\ref{Q-es-al-be}),
we get the desired result of Lemma \ref{Step2-2}.

\subsection{The proof of Lemma \ref{Step3-2}.}
\label{app-Step3-2}
Since $f(t, w) \in D(A^r)$, we can use part (5) of Lemma \ref{QG-es} to find an estimation
for (\ref{j12-N}) as follow
\begin{eqnarray}
\label{es-J12-N2}
J_{12}(N)
& \le &
   4 M^2 T^{2\rho}
   \int_{\lambda>N}
   \lambda^{2\gamma}
   \ud \dnorm{\es G_{\varphi, \alpha, \beta, A}(u)}^2
\nn
\\
&\le &
   4 M^2 T^{2\rho} N^{-2r}
   \int_{\lambda>N}
   \lambda^{2(\gamma+r)}
   \ud \dnorm{\es G_{\varphi, \alpha, \beta, A}(u)}^2
\nn
\\
&\le &
   E_1^2 N^{-2r}
   ,
\end{eqnarray}
where
$
   E_1
=
   2 M T^\rho
    \dnorm{G_{\varphi, \alpha, \beta, A}(u)}_{\gamma+r}
$.
We can repeat the proof of Lemma (\ref{Step2-2})
and use the estimation (\ref{es-J12-N2}) to obtain the result desired.

\section*{Acknowledgements}
 This research is funded by Vietnam National Foundation for Science and Technology Development (NAFOSTED) under grant number 101.02-2019.321.

\end{document}